\documentclass[11pt,leqno]{amsart}%

\usepackage[margin=1.1in]{geometry}
\usepackage[utf8]{inputenc}
\usepackage[T1]{fontenc}
\usepackage{times}
\usepackage{pdfsync}
\usepackage{dsfont}
\usepackage{color} 
\usepackage{amssymb}
\usepackage{bm} 
\usepackage{comment}
\usepackage{algorithm}
\usepackage{algpseudocode}
\usepackage{mathtools}
\usepackage{graphicx}
\usepackage{caption}
\usepackage{microtype}
\usepackage{enumerate}
\usepackage{subfig}
\usepackage{mathabx}
\usepackage{tikz, color}
\usetikzlibrary{matrix,arrows}

\usepackage{mleftright}
\mleftright
\usepackage{centernot}

\usepackage[hypertexnames=false]{hyperref}

 \newenvironment{listi}
  {\begin{list} 
 {(\roman{broj})}
{ \usecounter{broj}}
     \setlength{\labelwidth}{25pt}
  }
{   \end{list} }
\newcounter{broj}
 
\newenvironment{list1}
  {
  \begin{list}
 {\textsc{\arabic{broj1}.} }
 {\usecounter{broj1}
  \setlength{\itemindent}{0pt}
  \setlength{\listparindent}{0pt}
  \setlength{\itemsep}{-2pt}}
  \setlength{\labelwidth}{60pt}
  \setlength{\parsep}{0pt}
  }
  {\end{list}
  }
  \newcounter{broj1}

\newtheorem{theorem}{Theorem}[section]

\newtheorem{lemma}[theorem]{Lemma}

\newtheorem{proposition}[theorem]{Proposition}
\theoremstyle{remark}
\newtheorem{remark}[theorem]{Remark}
\newtheorem{example}[theorem]{Example}

\newcommand{\R}{\mathbb{R}}
\newcommand{\N}{\mathbb{N}}

\newcommand{\te}{\textrm}
\newcommand{\tacka}{\,\cdot\,}

\newcommand{\veps}{\varepsilon}
\DeclareMathOperator{\supp}{supp}

\DeclareMathOperator{\Lip}{Lip}
\DeclareMathOperator{\divv}{div}

\DeclareMathOperator{\Tan}{Tan}

\definecolor{mygreen}{rgb}{0.1,0.75,0.2}

\definecolor{darkgreen}{rgb}{0.1,0.6,0.2}
\definecolor{darkred}{rgb}{0.6,0,0}
\definecolor{lightgray}{rgb}{0.5,0.5,0.5}
\definecolor{turquoise}{rgb}{0.2,0.84,0.78}
\newcommand{\De}[1]{{\color{darkgreen}#1}}
\newcommand{\Yu}[1]{{\color{red}#1}}

\newcommand{\Ru}[1]{{\color{violet}#1}}
\usepackage{bbm}
\usepackage{mathrsfs}
\usepackage{galois}
\usepackage{float}
\usepackage{subfig}

\usepackage[shortlabels]{enumitem}
\usepackage{appendix}
\def\dd{\mathrm{d}}

\title{HV geometry for signal comparison}
\author{Ruiyu Han}
\address[Ruiyu Han and Dejan Slep\v{c}ev] {Department of Mathematical Sciences\\Carnegie Mellon University\\Pittsburgh, PA 15213}
\email{ruiyuh@andrew.cmu.edu, slepcev@math.cmu.edu}
\author{Dejan Slep\v{c}ev}
\author{Yunan Yang}
\address[Yunan Yang]
{Institute for Theoretical Studies\\
ETH Z\"urich\\
Z\"urich, Switzerland 8092}
\email{yunan.yang@eth-its.ethz.ch}
\thanks{}

\begin{document}
\date{\today}
\maketitle


{\footnotesize \emph{Dedicated to Bob Pego whose knowledge and love of mathematics greatly inspire us. }\par }

\begin{abstract}
In order to compare and interpolate signals, we investigate a Riemannian geometry on the space of signals. The metric allows discontinuous signals and measures both horizontal (thus providing many benefits of the Wasserstein metric) and vertical deformations. Moreover, it allows for signed signals, which overcomes the main deficiency of optimal transportation-based metrics in signal processing. We characterize the metric properties of the space of signals and establish the regularity and stability of geodesics. Furthermore, we introduce an efficient numerical scheme to compute the geodesics and present several experiments which highlight the nature of the metric. 
\end{abstract}




\section{Introduction}
In~\cite{MilYou01}, Miller and Younes introduced a deformation-based geometry on the space of $L^2$ functions in arbitrary dimension. It allows for and measures both horizontal deformations (as does the Wasserstein metric) and vertical deformations (as does the $L^2$ norm). Unlike the Wasserstein distance, it allows for signals that change signs. We will refer to the resulting geometry as the HV geometry throughout the paper.

Trouv\'e and  Younes \cite{TroYou05} provided some of the foundational results, and subsequent works generalizing the approach to a variety of distances and settings. However, there have been few works carefully studying the associated geometry of the signals and theoretically establishing its properties. This is in stark contrast with the abundance of developments of variants of optimal transportation. In the hope of fostering further development of this and related geometries, we provide a largely self-contained introduction in perhaps the simplest setting, namely one-dimensional signals. We then prove new results characterizing the resulting metric, providing various a priori estimates, and establishing the regularity and stability of minimizing geodesics. Furthermore, we introduce a simple and efficient scheme to compute the metric.

While this framework of~\cite{TroYou05} has been used in applications, primarily in image processing, we are unaware that it has been applied to the setting of one-dimensional signals. Here we argue that when considered in the space of signals (one-dimensional functions), it provides a viable metric that has desirable features and can be effectively computed. This opens the door for a variety of applications.
\medskip

\begin{figure}[h]
    \centering
    \includegraphics[width = 0.75\textwidth]{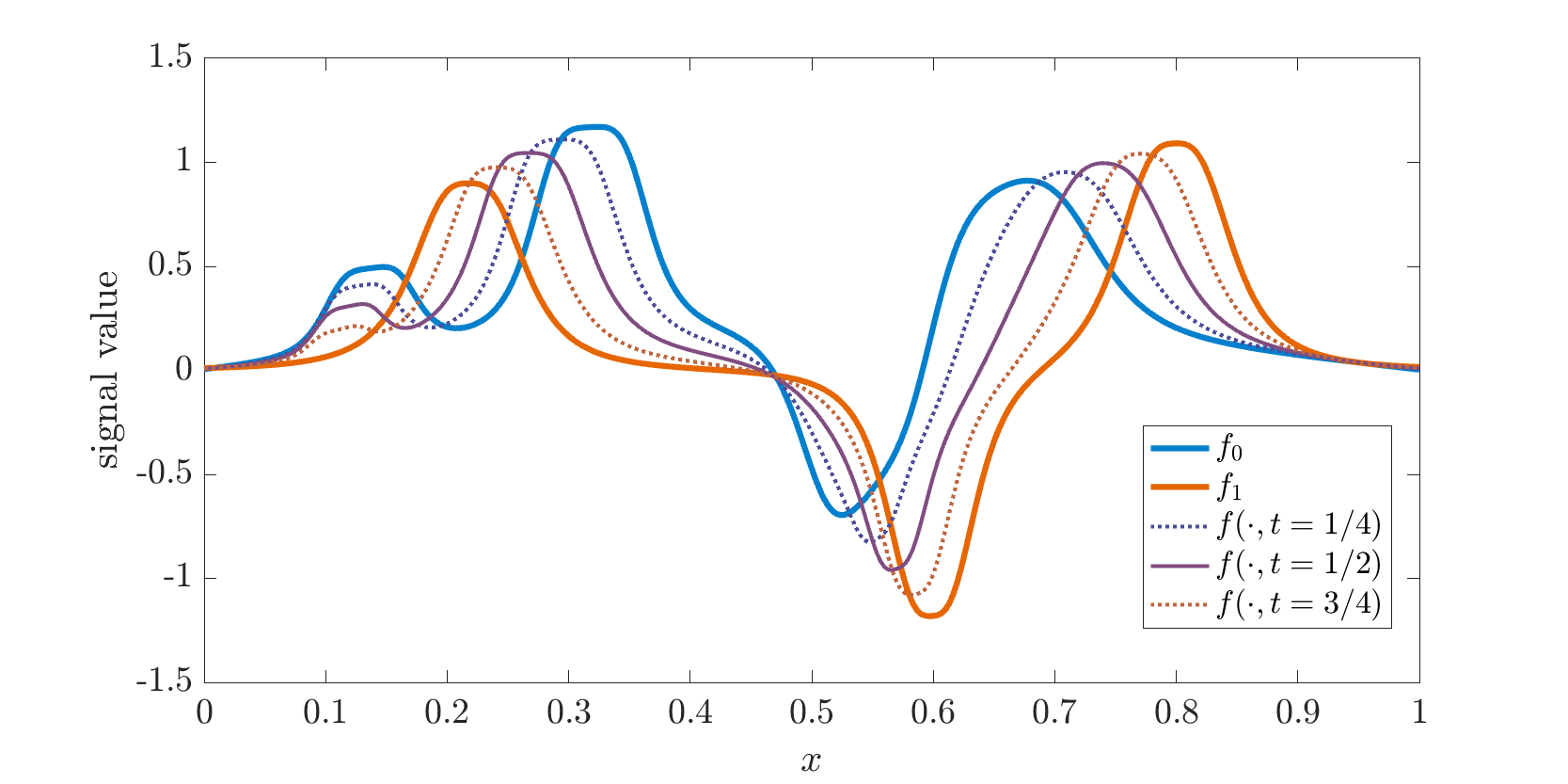}
    \caption{An example of the computed geodesic in the space of signals according to the HV geometry for $\kappa =0.1$, $\lambda=0.01$, and $\veps = 0.0005$. We note that for some features, horizontal transform dominates, while some parts are matched by vertically moving the signals.}
    \label{fig:c1}
\end{figure}

Given a finite interval, which, for simplicity, we set to be $[0,1]$, we consider the space of signals to be the $L^2$ functions. The paths on the space of signals are described by (weak) solutions of the transport equation with a source. In particular, a path connecting $f_0, f_1 \in L^2(0,1)$ is described by
\begin{align} \label{eq:path-intro}
\begin{split}
\partial_t f & = - \partial_x f \cdot v + z \quad \te{on } [0,1] \times [0,1],  \\
v(0, \tacka)& = v(1, \tacka)=0, \\
f(\tacka,0) & =f_0, \;\; f(\tacka,1)=f_1.
\end{split}
\end{align}
The action measuring the effort of deforming $f_0$ to $f_1$ along the particular path is 
\begin{equation} \label{eq:action1}
   A_{\kappa, \lambda, \veps}(f,v,z) =\frac12 \int_0^1 \int_0^1   \kappa v^2 + \lambda v_x^2  + \varepsilon v_{xx}^2 + z^2    \, \dd x \dd t
\end{equation}
where $\kappa >0$, $\lambda \geq 0$, and $\veps>0$. 

The $v^2$ term measures the horizontal movement of the signal, while the $z^2$ measures the vertical change. The parameter $\kappa$ is for modeling flexibility, as it accounts for horizontal and vertical variations differently. 
We note that the horizontal transport is modeled by the transport equation term, $\partial_x f \cdot v$, instead of the continuity equation $\divv (f v)$. Thus, even if $z \equiv 0$, the integral of $f$ is not preserved unless $\divv v \equiv 0$. The $\lambda v_x^2$ term measures how far from being conservative the transport is.
It is helpful to mention that in a higher dimensional analog of the metric $v_x$ term would be replaced by $\divv v$. 
In essence, it describes how much space needs to be created/removed to expand/contract the region where $f$ takes a certain value. The $\veps v_{xx}^2$ term is necessary for regularity. In particular, it ensures that very small regions cannot be inflated at an arbitrarily small cost, which would create an undesirable shortcut that could become the dominant transport mechanism. Indeed we discuss in Section \ref{sec:deg} that the geometry can degenerate if $\veps =0$. 

We define the distance between functions $f_0$ and $f_1$ as the infimum of the action among all admissible paths connecting them
\[ d_{HV( \kappa, \lambda, \veps)} := 
\inf_{(f,v,z)} \int_0^1 \sqrt{\frac12  \int_0^1   \kappa v^2 + \lambda v_x^2  + \varepsilon v_{xx}^2 + z^2    \, \dd x } \;  \dd t  \; = 
\inf_{(f,v,z)} \sqrt{ A_{\kappa, \lambda, \veps}(f,v,z)} . \]
The second equality is proved using reparameterization. When parameters $\kappa$, $\lambda$, and $\veps$ are clear in the context, we use the short-hand notation $d_{HV}$. See Figure~\ref{fig:c1} for an example of the computed geodesic between two signed signals based on the HV geometry.
\medskip

The metric we describe above is not new; modulo some minor technical details, it is the metric introduced by Miller and Younes \cite{MilYou01} and studied by Trouv\'e and  Younes  \cite{TroYou05m} and belongs to the family of metamorphoses studied by Trouv\'e, Younes, and  Holm  \cite{HoTrYo09}. 
It is also closely related to the metrics studied by Charlier, Charon, and Trouv\'e \cite{ChChTr16, ChChTr18}. 
In particular, \cite{TroYou05m} and \cite[Theorem 2]{ChChTr16} show the existence of geodesics, and  \cite[Property 1]{ChChTr18} shows the completeness of the metric. 
The works \cite{HoTrYo09, ChChTr16, ChChTr18} display very general frameworks and large families of possible deformation-based ways to compare functions on both fixed and changing sets and manifolds. 

Here we present a simple and self-contained description of the particular metric of interest in the space of signals.
Our problem description is primarily Eulerian instead of the approach based on groups of diffeomorphisms. The Eulerian point of view, allows us to consider paths between signals with less regularity. This allows us to prove that metrics that only penalize one derivative do not lead to a geodesic space, as minimizing geodesics may not exist; see Section \ref{sec:deg}.
We prove the existence of geodesics and the completeness of the resulting metric space of signals, which are analogous to those in \cite{TroYou05} and \cite{ChChTr18}, respectively.
Our approach allows us to characterize the convergence in the resulting space of signals as the $L^2$ convergence; see~Theorem \ref{conv_L2}. We are not aware of such a result in the other frameworks. In particular, knowing that the resulting metric is not weaker than the $L^2$ distance provides clarity about the distance that deformations induce. Furthermore, we prove a priori estimates on the minimizers, allowing us to exhibit that the minimizers preserve the regularity of the signals. In particular, Proposition \ref{prop:f_H1} shows the result for signals in $H^1$ and Remark  \ref{rem:fsmooth} shows it for smooth signals. 

Furthermore, we introduce a numerical algorithm, which iterates between optimizing the action over $v$ and $z$ with $f$ fixed and optimizing over $z$ and $f$ with $v$ fixed. Each of the subproblems amounts to minimizing a quadratic functional under a linear constraint (unlike the full problem where the constraint is nonlinear) and thus can be efficiently performed. 
We observe that the algorithm quickly (within a few iterations) converges to a local minimum of the action. 
By iterating between solving these two sub-problems, we can show that the action functional~\eqref{eq:action1} monotonically decreases on the continuous level. We employ a Lagrangian approach when optimizing over $z$ and $f$ with $v$ fixed, which involves numerical integration and numerical interpolation. When optimizing the action over $v$ and $z$ with $f$ fixed, we solve a fourth-order partial differential equation (PDE) through the finite-difference method. 
Since the geodesic obtained depends on the initial data, we also propose a  velocity initialization for which matches the $k$ most prominent peaks of each signal. To mitigate the issue of only finding a local minimizer, we optimize over $k\in \{0,1, \dots, k_{max} \}$ to select the initialization that leads to the smallest action value, where $k_{max}$ is the number we pick, denoting the maximum number of prominent peaks we aim to match. In particular, $k=0$ corresponds to zero-velocity initialization.

\subsection{Literature review.} 
Designing and choosing the right way to compare the functions considered are important in a number of contexts. 
In inverse problems, such as denoising, deblurring, or wave-form inversion, choosing an appropriate way to measure loss is crucial to the outcome. If the signals represent the distribution of mass where the location of the mass is more important than the density, optimal transportation metrics perform especially well \cite{EnFrYa16,engquist2022optimal}. In the setting of optimal transportation with quadratic cost, the equation in~\eqref{eq:path-intro} is replaced by the continuity equation:
\[ \partial_t f = - \divv(fv) \]
and the action is 
\[ A_W (f,v) = \int_0^1 \int_0^1 v^2 f \, \dd x \dd t.\]
The optimal transportation metrics have the requirement that the functions compared are nonnegative with the same total mass. The mass requirement was relaxed by the unbalanced optimal transport \cite{Liero_2017,unbalanced_1}, which allows for the change of the mass, but still requires non-negativity. For unbalanced transport, the continuity equation can have a multiplicative source term, and the action also penalizes signal amplification:
\begin{align*}
 \partial_t f & = - \divv(fv) + zf, \\
\te{and } \;\;  A_{UB} (f,v) & = \int_0^1 \int_0^1 (v^2 + z^2) f \, \dd x \dd t.
\end{align*}


There is a need to compare general signals with both positive and negative values in many applications, such as full-waveform inversion~\cite{engquist2022optimal}. It poses a limitation for applications of optimal transportation. Several ideas for 
 adapting optimal transportation approaches to signed signals have been introduced~\cite{engquist2022optimal}. They include adding a constant, 
exponentiating the signal, comparing positive and negative parts separately, and considering the transportation between the graphs of functions on the product space (e.g., $TL^p$ metric \cite{TPKRS17}). Comparing the signals on graphs has been shown as effective in full-waveform inversion and related problems \cite{MBMO19}.
However, all of these approaches lack one of the key elements, namely the Riemannian nature that allows for interpolation or the desired invariance  $d(f_0, f_1) = d(-f_0, -f_1)$.
\medskip

Geometric ways for comparing shapes and submanifolds of $\R^d$ have received a lot of attention in computer vision, computational anatomy and geometry processing. Riemannian geometries on the space of curves have been much studied. Mumford and Michor~\cite{MicMum06} showed that if only the $L^2$ norm of the velocity is considered in the action, then the metric degenerates.
 Bruveris, Michor and Mumford \cite{BrMiMu14, Bruveris15} 
 and Nardi, Peyre, and Vialard \cite{NaPeVi16} 
 established that the space of immersed plane curves is geodesically complete when the Sobolev or, respectively, $ BV$-based metric involves the $L^2$ norm of two or more derivatives. 
Early works on computational anatomy and nonrigid registration have led to studies of geometries of the shapes of surfaces, as well as higher dimensional manifolds, many of which are surveyed in the book by Younes~\cite{Younes10}. Recent results~\cite{BaBrHaMi22, BaHaMi20} established the local well-posedness of geodesic equations when the metrics penalize at least one derivative of the deformation field. 

A different line of work investigates the spaces of shapes considered as interiors of sets. In some sense, this takes into account the mass contained in the set, as does the Wasserstein distance. Liu, Pego and one of the authors, \cite{LPS19}, showed that restricting the Wasserstein geometry to the space of characteristic functions is not viable, as the geodesics do not exist. In fact, minimizing curves converge weakly to the Wasserstein geodesics in the unconstrained space. 
Wirth,  Bar, Rumpf, and Sapiro~\cite{WBRS11} considered actions that penalize the gradient of the velocity field and showed that, in this case, the geometry is viable. The framework was refined by Rumpf and Wirth~\cite{RumWir13,RumWir15}, who also provided a numerical method for finding geodesics.

In parallel with the studies of the geometry of the spaces of curves and shapes, researchers considered the metrics which allow for comparing signals, such as gray-level images, while allowing for both deformations of the domain and intensity variations. Trouv\'e \cite{Trouve95} introduced the basic description based on group actions. This was refined into the description of deformable templates \cite{TroYou05} and metamorphoses \cite{TroYou05m} by Trouv\'e and Younes. As we indicated at the beginning, the geometry we study belongs to this family. Eulerian description of metamorphosis was carried out in \cite{HoTrYo09}. 
Charlier, Charon, and Trouv\'e \cite{ChChTr16,ChChTr18} considered spaces where one compares manifolds together with functions on them. The authors of \cite{ChChTr16,ChChTr18} showed that under sufficient regularity assumptions, in the space of manifold, function pairs are geodesically complete. 
Many of the numerical approaches (e.g. \cite{MiTrYo06,FGG21num}) to computing the minimizing geodesics between the given images relied on the so called shooting methods where one iterates computing the  forward geodesic for the given initial conditions and adjusts the initial conditions. 
Berkels, Effland, and Rumpf \cite{DeEfRu15num} took a different approach and gave a variational formulation based on discrete-in-time paths where one minimizes the deformation between consecutive images in a way that is consistent with the action integrated in time. This is a promising and well founded approach. One difference to our approach is that their sub-problems remain nonconvex. 

\subsection{Outline}
The rest of the paper is organized as follows. In Section~\ref{sec:main}, we define the HV geometry and the associated distance and establish their properties. In particular, we present the scaling invariances of the metric in Proposition \ref{prop:properties}. In Section \ref{sec:tangent}, we rigorously identify the tangent space. The identification is analogous but slightly different from the one \cite{TroYou05} in that instead of equivalence classes, we identify the representatives that achieve the minimal length (just as the gradient vector fields minimize the action for Wasserstein geometry). In Section \ref{subsec:exist}, we provide a number of a priori estimates, prove the existence of geodesics (in a slightly different way from that in \cite{TroYou05}), and obtain representation formulas satisfied by geodesics. In Section \ref{sec:complete}, we prove the completeness of the $d_{HV}$ metric on $L^2(0,1)$ and that the topology induced by $d_{HV}$ on $L^2(0,1)$ is the same as the one induced by the $L^2$ norm. 
Moreover, we show in Proposition \ref{prop:H1Vdeg} that without involving the second-order derivative of the velocity, action minimizers may not exist.
In Section \ref{sec:further_properties}, we further study several properties of the geodesics. In Proposition \ref{prop:f_H1}, we show that if the starting and the ending signals are in $H^1$, then the signals along the minimizing geodesic remain in $H^1$. Furthermore, Remark \ref{rem:fsmooth} indicates that if the signals are smooth, so is the geodesic connecting them. Proposition \ref{prop:stability_geod}  establishes the $L^2$ stability of minimizing geodesics with respect to perturbations of the endpoints. We note that this is not a straightforward compactness result since, initially, one only has control over the $L^2$ norms of the signals. Section 
\ref{sec:EL} is devoted to establishing first-order optimality conditions of the action minimizers. We first obtain these conditions for general $L^2$ signals. If the signals are further in $H^1$, we show that the Euler--Lagrange equations have the form of differential equations. 
In Section \ref{sec:scheme}, we introduce our numerical method. We first present the iterative minimization scheme that uses two sub-problems which are based on Euler--Lagrange equations and representation formulas for the geodesics. The discretization we use based on finite difference is included in Section \ref{sec:num-discrete}. We describe our approach to the path initialization in Section \ref{subsec:init}. As the energy often has local minimizers, it is important to use initialization informed by the signals. In Section~\ref{sec:experiments}, we display several illustrative examples that highlight the properties of the $HV$ signal geometry,  as well as the numerical approach we take. We also  apply our scheme to signals from the ECG datasets and seismology. Finally we discuss the parameter selection in Section \ref{subsec:para}.


\section{HV Geometry and its properties.} \label{sec:main}

To rigorously define the HV distance we start by precisely defining the set of admissible paths.
Let $\mathcal V:= L^2((0,1),H^2(0,1) \cap H^1_0(0,1))$. Given $f_0, f_1 \in L^2(0,1)$ we define the set of admissible paths to be  
\begin{align} \label{eq:A}
\begin{split}
\mathcal A(f_0, f_1):= \big\{(f,v,z)\:|\: &f \in L^{2}((0,1),L^2(0,1)) \cap C([0,1],H^{-1}(0,1)),\\
&v\in \mathcal V, \; z\in L^2((0,1),L^2(0,1)),\\
&\partial_t f+\partial_x f\cdot v=z \text{ weakly, }\:f(\cdot,0)=f_0, f(\cdot,1)=f_1
\big\}.
\end{split}
\end{align}
By weak solutions of $\partial_t f+ \partial_x f\cdot v=z$ above we mean that 
 for any test function $\phi\in C^{\infty}([0,1]^2)$,
 
\begin{align}  \label{eq:weaksol}
\begin{split}
- \int_0^1\int_0^1f\partial_t\phi & \dd x\,\dd t+\int_0^1 f_{1}(x)\phi(x,1)\,\dd x-\int_0^1 f_{0}(x)\phi(x,0)\,\dd x\\
& =\int_0^1\int_0^1  f\phi\partial_x v\,\dd x\,\dd t
+\int_0^1 \int_0^1f\partial_x \phi v\,\dd x\,\dd t
+\int_0^1 \int_0^1\phi z\,\dd x\,\dd t.
\end{split}
\end{align}
The $HV$ distance is then defined by
\begin{equation} \label{eq:dHV}
d_{HV( \kappa, \lambda, \veps)} := 
\inf_{(f,v,z) \in \mathcal A(f_0,f_1)} \int_0^1 \sqrt{\frac12  \int_0^1   \kappa v^2 + \lambda v_x^2  + \varepsilon v_{xx}^2 + z^2    \, \dd x } \;  \dd t 
\end{equation}

The definition of the distance implies the following simple, but useful properties.

\begin{proposition} \label{prop:properties}
Consider $f_0, f_1 \in L^2(0,1)$. Let $c>0$. Then
\begin{listi}
\item $d_{HV}(f_0, f_1) \leq \| f_0  - f_1 \|_{L^2}$.
\item $\displaystyle{d_{HV}(-f_0, -f_1) =  d_{HV}(f_0, f_1)}$
\item $\displaystyle{d_{HV}(f_0+c, f_1+c)  = d_{HV}(f_0, f_1)}$
\item $\displaystyle{d_{HV( c^2 \kappa, c^2\lambda, c^2\veps)}(c f_0, c f_1)   = c d_{HV( \kappa,\lambda, \veps)}(f_0, f_1) }$
\end{listi}
To indicate the behavior of the action with respect to rescaling the space extend $f_0$ and $f_1$ periodically to $\R$. Likewise,  given a path $(f,v,z)$ consider it extended periodically to $\R$. Then for $L \in \N$
\begin{listi}
\addtocounter{broj}{4}
\item  $\displaystyle{A_{ L^2\kappa, \lambda, \veps/L^2}(f(L\tacka,\tacka),v(L\tacka,\tacka),z(L\tacka,\tacka)) =  A_{\kappa,\lambda, \veps}(f,v,z)},\;$
where the action is considered only on $[0,1]$, as usual.
\end{listi}
\end{proposition}
Property (i) is proved by computing the action of the linear interpolation while the remaining properties are proved by a straightforward application of the definition of the action. 


\subsection{Degeneracy without the second derivatives} \label{sec:deg}

Before proving rigorous results about $d_{V}$, we show that is we
 $\varepsilon = 0$ the geometry of signals would not have all of the desirable properties. In particular the show that  there exist functions $f_0$ and $f_1$ such that there does not exist any minimizers for the action. 
 
 To be more precise, consider the set of admissible paths to be as in \eqref{eq:A}, but with $\mathcal V$ replaced by $ L^2((0,1), H^1_0(0,1))$. 
 Consider the action 
 \begin{equation} \label{eq:ac1}
 A(f, v, z) = \int_0^1 \int_0^1 v^2 + v_x^2 + z^2 d x d t. 
\end{equation}
 
 We start by noting that linear interpolation is not  optimal when $f_0$ and $f_1$ are constant functions that differ sufficiently. 
\begin{lemma} \label{lem:aux1}
If $\varepsilon = 0$, then for all $\lambda \in [0, \infty)$ there exists $H \in \mathbb{R}$ such that the linear interpolation path between $f_0 \equiv 0$ and $f_1 \equiv H$ is not optimal. 
\end{lemma}
\begin{proof}
Consider $f_0 \equiv 0$ and $f_1 \equiv H$ for some $H >0$. One possible path between $f_0$ and $f_1$ is by $L^2$ interpolation. This path is defined by $v \equiv 0$ and $z \equiv H$. Then 
\[ A(f, v, z) = \int_0^1 \int_0^1 z^2 d x d t = H^2 \]

We construct a competitor with  nonzero velocity $v$. In the time interval $t \in [0, 1/3]$, use an $L^2$ interpolation between $f_0$ and 
\[ f_2 (x) = \begin{cases}
H & x \leq s \\
0 & x > s
\end{cases} \]
for fixed  $s \in (0, 1/2)$. In this time interval, $v \equiv 0$, and $\int_0^{\frac{1}{3}} \int_{0}^s z^2 d x d t = 3 s H^2.$
On the time interval $t \in [1/3, 2/3]$, the signal moves from $f_2$ to 
\[ f_3 (x) = \begin{cases}
H & x \leq 1 - s \\
0 & x > 1 - s
\end{cases} 
\]
by moving the edge of the step function with a constant speed $3 (1 - 2 s)$. 
The velocity parameterized by the initial position is $w(x) = 3(1-2s) \frac{x}{s}$ for $x \in (0,s]$ and $w(x) = 3(1-2s) \frac{1-x}{1-s}$ for $x \in (s,1)$. 
Then for all $t \in [1/3, 2/3]$, the velocity is $v\left(x+ w(x)\left( t - \frac13 \right),t \right) = w(x)$. 
By chain rule $0 \leq v \leq 3 (1 - 2 s)$ and $|v_x| \leq \frac{ 3 (1 - 2 s) }{s}$. So 
\[ \int_{\frac{1}{3}}^{\frac{2}{3}} \int_{0}^{1} v^2 + \lambda v_x^2 d x d t \leq 3 (1 - 2s)^2 \left( 1 + \frac{\lambda}{s^2} \right) \]

The interpolation between $f_2$ and $f_3$ is that the signal is constant along the trajectories. 
Finally use $L^2$ interpolation between $f_3$ and $f_1$. This is symmetric to the time interval $[0, 1/3]$ and adds $3 s H^2$ to the action. 

Thus the total action for this path is bounded above by 
$ 6 s H^2 + 3 (1 - 2s)^2 \left( 1 + \frac{\lambda}{s^2} \right).$
By picking $s < \frac{1}{6}$, and then $H>0$ such that 
$H^2 > \frac{3}{1 - 6s} (1 - 2s)^2 \left( 1 + \frac{\lambda}{s^2} \right)$,
we have that 
$ H^2 > 6 s H^2 + 3 (1 - 2s)^2 \left( 1 + \frac{\lambda}{s^2} \right)$. Therefore
 $L^2$ interpolation is not the optimal path between $f_0 \equiv 0$ and $f_1 \equiv H$. 
\end{proof}

\begin{proposition} \label{prop:H1Vdeg}
There exists $H >0 $ such that there is no  path between $f_0 \equiv 0$ and $f_1 \equiv H$ minimizing the action \eqref{eq:ac1} . 
\end{proposition}
\begin{proof}
Assume $H$ satisfies the condition of Lemma \ref{lem:aux1}.
Consider any path $(f, v, z)$ from $f_0$ to $f_1$. 
We split the argument into two cases depending on whether $v \equiv 0$. 

\emph{Case $1^o$} If $v \equiv 0$. 
Then  $ f_t = z$ for all $x,t$, that is the path is linear interpolation. 
Hence, be Lemma \ref{lem:aux1} the path does not minimize the action. 

\emph{Case $2^o$} $v \not \equiv 0$. 
From the path $(f, v, z)$ we can construct a new path $(\tilde f, \tilde v, \tilde z)$ by creating two copies of $f$ shrank to interval $\frac12$. 
\begin{align*}
    \tilde f (x, t) &= \begin{cases}
    f(2x, t) & x \leq \frac{1}{2} \\
    f(2x-1, t) & x \geq \frac{1}{2} 
    \end{cases} \\
    \tilde v (x, t) &= \begin{cases}
    \frac 12 v(2x, t) & x \leq \frac{1}{2} \\
    \frac 12 v(2x-1, t) & x \geq \frac{1}{2} 
    \end{cases} \\
    \tilde z (x, t) &= \begin{cases}
    z(2x, t) & x \leq \frac{1}{2} \\
    z(2x-1, t) & x \geq \frac{1}{2}.
    \end{cases} \\
\end{align*}
A consequence of this is that 
\[ \tilde v_x (x, t) = \begin{cases}
    v_x (2x, t) & x \leq \frac{1}{2} \\
    v_x (2x-1, t) & x \geq \frac{1}{2}.
    \end{cases} \]
For the action of this path, 
\begin{align*}
    A(\tilde f, \tilde v, \tilde z) 
    &= \int_{0}^{1} \int_0^{\frac{1}{2}} \tilde v^2 + \tilde v_x^2 + \tilde z^2 d x d t + \int_{0}^{1} \int_{\frac{1}{2}}^1 \tilde v^2 + \tilde v_x^2 + \tilde z^2 d x d t \\
    &= \int_{0}^{1} \int_0^{\frac{1}{2}} \frac{1}{4} v(2x, t)^2 + \tilde v_x(2x, t)^2 + \tilde z(2x, t)^2 d x d t \\
    &\quad+ \int_{0}^{1} \int_{\frac{1}{2}}^1 \frac{1}{4} \tilde v(2x - 1, t)^2 + \tilde v_x(2x - 1, t)^2 + \tilde z(2x - 1, t)^2 d x d t \\
    &= \int \int \frac{1}{4} v^2 + v_x^2 + z^2 d x.
\end{align*}
Since $v$ is non-zero, this is strictly less than the action $A(f, v, z)$. Thus the path $(f, v, z)$ is not minimal. 

Therefore  so no path between $f_0$ and $f_1$ minimizes the action. 
\end{proof}

\subsection{Identification of the tangent space.} \label{sec:tangent}

We note that $d_{HV}$, defined in \eqref{eq:dHV}, can be seen as the length of a curve in the space of signals. Indeed, as we show below, $(L^2(0,1), d_{HV})$ is geodesic space. 
From the definition of the admissible curves and the action we see that in Eulerian description of the tangent space
\begin{equation}
    \Tan_{f,Euler}= \{ -v f_x  + z \::\: v \in H^2(0,1) \cap H^1_0(0,1), \;\; z \in L^2(0,1) \}\; \subseteq H^{-1} 
\end{equation}
where $v f_x $ is the element of $H^{-1}$ defined by $\langle -vf_x, \phi \rangle = 
\int_0^1 v_x f \phi + vf \phi_x \dd x$ for all $\phi \in H^1_0$. 
We show below that in the Lagrangian description the tangent space can be identified with a subspace of $(H^2(0,1) \cap H^1_0(0,1)) \times L^2(0,1)$. 

\begin{lemma}  \label{lem:var_tangent}
Given $f \in L^2(0,1)$, $\bar v \in H^2(0,1) \cap H^1_0(0,1)$ and $\bar z \in L^2(0,1)$ there exists a unique pair $v \in H^2(0,1) \cap H^1_0(0,1)$, $ z \in L^2(0,1)$ minimizing the instantaneous action 
\begin{equation} \label{eq:Q}
Q_f(v,z) = \frac12 \int_0^1   \kappa v^2 + \lambda v_x^2  + \varepsilon v_{xx}^2 + z^2    \, \dd x 
\end{equation}
under the constraint
\begin{equation} \label{eq:C}
    \forall \phi \in H^1_0(0,1) \qquad \quad 
    \int_0^1 f v_x \phi + fv \phi_x + z \phi \,\dd x =  \int_0^1 f \bar v_x \phi + f \bar v \phi_x + \bar z \phi \, \dd x. \qquad \;
\end{equation}
We will denote the solution mapping by $S : H^2(0,1) \cap H^1_0(0,1) \times L^2(0,1) \to H^2(0,1) \cap H^1_0(0,1) \times L^2(0,1) $, that is $(v,z) = S(\bar v, \bar z)$. 
\end{lemma}
\begin{proof}
Let $C(f, \bar v, \bar z)$ be the set of $v \in H^2(0,1) \cap H^1_0(0,1)$, $ z \in L^2(0,1)$ satisfying the constraint \eqref{eq:C}. Note that if 
$(v_n, z_n) \in C(f, \bar v, \bar z)$, $v_n \to v$ in $H^1_0$, $z_n \rightharpoonup z $ in $L^2$ as $n \to \infty$ and $v \in H^2(0,1)$ then $(v,z) \in C(f, \bar v, \bar z)$. Namely $f \phi \in L^2$ ensures that $\int f(v_n)_x \phi \dd x \to\int f v_x \phi \dd x $ and  $v_n \to v$ in $L^\infty$ implies 
$\int f v_n \phi_x \dd x \to\int f v \phi_x \dd x $ as $n \to \infty$.

We note that $(\bar v, \bar z ) \in C(f, \bar v, \bar z)$ and $Q_f(\bar v, \bar z) < \infty$. Let $(v_n, z_n) \in C(f, \bar v, \bar z)$ be a minimizing sequence of $Q$. Thus $\{v_n\}_n$ is bounded in $H^2$ and $\{z_n\}_n$ is bounded in $L^2$. Hence there is a 
subsequence, which by relabeling we can assume to be the whole sequence, $v_n \to v$ in $H^1_0$ and $z_n \rightharpoonup z$ in $L^2$ to some $v \in H^2 \cap H^1_0$ and $z \in L^2$. By above $(v,z) \in C(f, \bar v, \bar z)$. Since $Q_f$ is sequentially lower-somicontinuous with respect to $H^1$ convergence in $v$ and weak $L^2$ convergence in $z$ we conclude that $(v,z)$ minimizes $Q_f$  over $C(f, \bar v , \bar z)$. 

The uniqueness follows from the fact that $Q_f$ is strictly convex and that $C(f, \bar v, \bar z)$ is convex. 
\end{proof}

We now characterize the solution $(v,z) = S(\bar v, \bar z)$ above via the first variation. We note that condition \eqref{eq:C} implies that $(v-\bar v) f_x$ belongs to $L^2$, in the sense that $\langle (v-\bar v) f_x , \phi \rangle = \int_0^1 (z-\bar z) \phi$ for all $\phi \in H^1_0(0,1)$, where $\langle \tacka, \tacka \rangle$ denotes the dual pairing between $H^{-1}$ and $H^1_0$. It is straightforward to show that for $f \in L^2(0,1)$ and $u \in H^2(0,1) \cap H^1_0(0,1)$ the condition  $u f_x \in L^2(0,1)$ is equivalent to $uf \in H^1(0,1)$. 
This motivates us to introduce the space
\begin{equation} \label{eq:R}
    R(f) = \left\{ u \in H^2(0,1) \cap H^1_0(0,1) \::\: uf \in H^1(0,1) \right\}.
\end{equation}
Note that for $u \in R(f)$ $\; u f_x = (uf)_x - u_x f \in L^2$. 
We also remark that if $f \in H^1$ then $R(f) = H^2(0,1) \cap H^1_0(0,1)$. 

\begin{lemma} \label{lem:first_var_gen}
Consider $f \in L^2(0,1)$, $\bar v \in H^2(0,1) \cap H^1_0(0,1)$ and $\bar z \in L^2(0,1)$. Then $(v,z) = S(\bar v, \bar z)$ if and only if $v \in H^2(0,1) \cap H^1_0(0,1)$, $z \in L^2(0,1)$ and 
\begin{equation} \label{eq:first_var_gen}
\forall u \in R(f) \qquad \;
\int_0^1 \kappa v u + \lambda v_x u_x + \veps v_{xx} u_{xx} + ( (v - \bar v)f_x + \bar z)(u f_x) \dd x = 0.  
\end{equation}
\end{lemma}
\begin{proof}
If  $(v,z) = S(\bar v, \bar z)$ then $v \in H^2(0,1) \cap H^1_0(0,1)$, $z \in L^2(0,1)$. Taking $u \in \R(f)$ and $h = u f_x $, which belongs to $L^2$, since $u \in R(f)$ we note that $(v + s u, z + sh) \in C(f, \bar v, \bar z)$. First variation in of $Q_f$ at $(v,z)$ gives the condition \eqref{eq:first_var_gen}. 

Now assume that $v \in H^2(0,1) \cap H^1_0(0,1)$, $z \in L^2(0,1)$ and that \eqref{eq:first_var_gen} holds. Let $(\tilde v, \tilde z) = S(\bar v, \bar z)$. By above \eqref{eq:first_var_gen} holds for $(\tilde v, \tilde z)$ in place of $(v,z)$. Furthermore note that $v - \tilde v = (v - \bar v) - (\tilde v - \bar v)  \in R(f)$. Taking $u = v - \tilde v$ and subtracting the two forms of \eqref{eq:first_var_gen} gives
\[ \int_0^1 \kappa (v - \tilde v)^2   + \lambda (v_x - \tilde v_x)^2 + \veps (v_{xx} - \tilde v_{xx})^2 + ( (v - \tilde  v)f_x)^2  \dd x =0\]
Thus $v = \tilde v$ and $z = \tilde z$.
\end{proof}

It is useful to note that if $f \in H^1(0,1)$ and hence $R(f)  = H^2(0,1) \cap H^1_0(0,1)$, the condition \eqref{eq:first_var_gen} means that for $g = - f_x \bar v + \bar z$,   $\;v$ is a weak solution of 
\begin{align}
\begin{split}
 \veps v_{xxxx} -\lambda v_{xx} + \kappa v + f_x^2 v & = -  f_x g   \quad \te{on } (0,1) \\
 v & = 0 \quad \te{at } \{0,1\} \\
 v_{xx} & = 0 \quad \te{at } \{0,1\}.
 \end{split} 
\end{align}

The above lemmas allow us to characterize the tangential velocities, namely we define
\begin{equation} \label{eq:tan_space}
    \Tan_f = \left\{ (v,z) \in H^2(0,1) \cap H^1_0(0,1) \times  L^2(0,1) \;:\: (v,z) = S(v,z) \right\}. 
\end{equation}
Lemma \ref{lem:first_var_gen} gives that a pair $(v,z) \in H^2(0,1) \cap H^1_0(0,1) \times  L^2(0,1)$ belongs to $\Tan_f$ if and only if
\begin{equation} \label{eq:cond_tan}
    \forall u \in R(f) \qquad \;
\int_0^1 \kappa v u + \lambda v_x u_x + \veps v_{xx} u_{xx} +  z u f_x \dd x = 0. 
\end{equation}

In the special case that $f \in H^1(0,1)$ we can characterize 
$\Tan_f$ as the set of pairs $(v,z)$ where  $z \in L^2(0,1)$ and $v \in H^2(0,1) \cap H^1_0(0,1)$ is a weak solution of 
\begin{align} \label{eq:cond_tanH1}
\begin{split}
 \veps v_{xxxx} -\lambda v_{xx} + \kappa v  & = -  f_x z   \quad \te{on } (0,1) \\
 v & = 0 \quad \te{at } \{0,1\} \\
 v_{xx} & = 0 \quad \te{at } \{0,1\}.
 \end{split} 
\end{align}

We furthermore remark that if $u \in R(f)$ then, since $H^1$ functions in one dimension are continuous we have that both $uf$ and $u$ are continuous. Thus $f= uf/u$ must be continuous on the set where $u \neq 0$. In other words if $f$ does not have a representative that is continuous at some $x$ then all functions $u \in R(f)$ are equal to zero at $x$; $u(x)=0$. 
In particular if $f$ in nowhere continuous then the only $u$ in $R(f)$ is the zero function. Consequently if $f$ is nowhere continuous then $\Tan_f = H^2(0,1)\cap H^1_0(0,1) \times L^2(0,1)$, that  every pair is a tangent pair.





\subsection{Existence of minimizing paths}\label{subsec:exist}
Here we show the existence of minimizers of the action \eqref{eq:action1}


\begin{theorem}\label{th:minimizer_existence}
Consider  $f_0,f_1\in L^{2}(0,1)$. 
There exists an admissible path $(f,v,z) \in \mathcal A(f_0, f_1) $ minimizing the action \eqref{eq:action1}.
\end{theorem}                           

Before proving Theorem \ref{th:minimizer_existence}, we first establish several  properties of solutions of the transport equation. These results and the techniques are standard, we present them for completeness. 

\begin{lemma}\label{lem:representation_sol}
 Given $v\in\mathcal V$, $z\in L^2((0,1)^2)$. Let $f_0\in L^2(0,1)$. If $f\in L^{2}((0,1)^2)$ is a weak solution to the initial value problem
 \[
 f(\cdot,0)=f_0,\quad\partial_t f+\partial_x f\cdot v=z \quad \;\te{on } [0,1]^2,
 \]
 in the sense that for every $\phi\in C^{\infty}([0,1]^2)$, \eqref{eq:weaksol} holds, then
$f$ has the following representation: for  a.e. $x,t$ in $[0,1]^2$,
\begin{equation}\label{eq:solf}
f(\Phi(x,t),t)=f_0(x)+\int_0^t z(\Phi(x,s),s)\,\dd s, 
\end{equation}
where $\Phi$ is the flow of the vector field $v$:
\begin{equation}\label{eq:flow-map}
\partial_t \Phi(x,t)=v(\Phi(t,x),t),\quad \Phi(0,x)=x.
\end{equation}
It follows that the weak solution $f$ is unique. 
 \end{lemma}
 
 \begin{proof}
Let us extend $f_0,z,v,f$ by zero outside of $[0,1]$, and denote the extensions by $\tilde z,\tilde f_0,\tilde v,\tilde f$. Note that $\tilde v \in L^2(0,1, W^{1, \infty}(\R))$ since $H^2(0,1)$ embeds in $W^{1, \infty}(0,1)$, $v(0,\tacka) \equiv 0 $, $v(1,\tacka) \equiv 0$, and extending by zero preserves the Lipschitz constant. 

 Then $\tilde f\in L^{2}(0,1,L^2(\mathbb R))$ satisfies that for all test functions $\phi\in C^{\infty}(\mathbb R\times [0,1])$ with compact support in $\mathbb R\times [0,1)$
\[-\int_0^1 \int_{\mathbb R}\tilde f\partial_t\phi\,\dd x\dd t -\int_{\mathbb R}\tilde f_0 \phi(x,0)\,\dd x=\int_0^1\int_{\mathbb R} \tilde f(\tilde v\partial_x\phi+\phi\partial_x \tilde v)\,\dd x\dd t+\int_0^1\int_0^1 \phi \tilde z\dd x\dd t.\]
Since $\tilde v$ is Lipschitz in space, the representation of solutions of the transport equation (see Proposition 2.3 in \cite{Amb14}), 
for  a.e. $x$
\[
\tilde f(\tilde \Phi(x,t),t)=\tilde f_0(x)+\int_0^t \tilde z(\tilde \Phi(x,s),s)\,\dd s, \quad \; t\in [0,1],
\]
where $\tilde \Phi$ is the flow of the vector field $\tilde v$, which satisfies the differential equation
\[
\partial_t \tilde \Phi(x,t)=\tilde v(\tilde \Phi(x,t),t),\quad \tilde \Phi(0,x)=x
\]
for each $t$.  
Notice that since $\tilde v(x,t)=0$ for all $x \in \R \setminus (0,1)$ and all $t\in [0,1]$, by the uniqueness of flow map we have that
for a.e. $x \in (0,1)$, $\tilde \Phi(x,t) \in (0,1)$ for all $t$. 
Define $\Phi(x,t):=\tilde \Phi(x,t)\mathbbm{1}_{\{x\in [0,1]\}}$.
Then $\Phi$ is the flow of $v$ and 
for a.e. $x$ in $[0,1]$,
\[
f(\Phi(x,t),t)=f_0(x)+\int_0^t z(\Phi(x,s),s)\,\dd s, \quad\forall t\in [0,1].
\]
\end{proof}

\begin{remark}\label{rem:zformula} [Representation formulas for action minimizing paths.] Note that if $z \in L^2((0,1)^2)$ is a weak solution of $\partial_t z + \partial_x(zv) = 0$ and $v$ satisfies the conditions of the previous lemma, then it is a weak solution of $\partial_t z + (\partial_x z) v= -z \partial_x v$. Thus 
\[ z( \Phi ( x, t) , t) = z( x, 0 ) - \int_0^t z( \Phi ( x, s) , s) 
 v_x ( \Phi(x, s), s) \dd s \]
which is an integral form of the ODE whose solution is 
\begin{equation} \label{eq:solz}
 z( \Phi ( x, t) , t) = z( x, 0 )\, e^{-\int_0^t v_x ( \Phi(x, s), s) \dd s}, 
\end{equation}
where $\Phi$ is the flow of $v$. Here we emphasize that by $v_x$ we always denote the partial derivative of $v$ with respect to the first variable, and not the derivative of the composition.

Let $J(x,t) = e^{-\int_0^t v_x ( \Phi(x, s), s) \dd s}$. We have the following representation:
\[
f_1(\Phi(x,1))=f_0(x)+\int_0^1z( \Phi ( x, t) , t)\,\dd t=f_0(x)+z(x,0) \int_0^1 \,J(x,t) \,\dd t.
\]
where the first equality follows from Lemma \ref{lem:representation_sol}. This allows us to determine $z(x,0)$ and hence
\begin{equation} \label{eq:zJ}
 z( \Phi ( x, t) , t) = \left( f_1 ( \Phi(x, 1) ) - f_0 (x) \right) \, \frac{J(x,t)}{\int_{0}^{1} J(x,\tau)  \dd \tau}. 
\end{equation}
Hence we have the following formula for $f( \Phi ( x, t) , t)$ 
\begin{align}
f(\Phi(x,t),t) &= f_0(x)+\int_0^t z(\Phi(x,s),s)\,\dd s = f_0(x) + ( f_1 ( \Phi(x, 1) ) - f_0 (x)) \,  \frac{\int_0^t J(x,s) \dd s}{\int_0^1 J(x,s) \dd s} \nonumber \\
 &= \left(1-\eta(x,t) \right) \,  f_0( x) +  \eta(x,t) \, f_1 ( \Phi ( x, 1)), \quad \te{where } \eta(x,t) =  \frac{\int_0^t J(x,s) \dd s}{\int_0^1 J(x,s) \dd s}. \label{eq:f-formula}
\end{align}
Equation~\eqref{eq:f-formula}  implies that at any $t\in [0,1]$ and a fixed $x \in [0,1]$, $f ( \Phi ( x, t), t ) $ is an interpolation between the initial condition $f_0(x) = f_0(\Phi ( x, 0))$ and the final-time condition  $f_1 (\Phi ( x, 1))$, with a time-dependent weight function $\eta(x,t)$ determined by the velocity $v$.

We note that if $(f,v,z) \in \mathcal A(f_0,f_1)$ then $f$ is a weak solution of $f_t + f_x v =z$. We show in Section \ref{sec:EL}, equation \eqref{eq:ELz1} that if the path is a critical point of the action then 
$z$ solves the continuity equation $z_t + (zv)_x=0$. The formulas \eqref{eq:f-formula} and \eqref{eq:zJ} provide the Lagrangian representation formulas for  $f$ and $z$.

\smallskip

\end{remark}
 \smallskip

\begin{lemma}\label{f_n_energy_est}
Let $f_0\in L^2(0,1)$, $v\in L^2(0,1, H^2(0,1) \cap H^1_0(0,1))$, $z\in L^2((0,1)^2)$.  If $f$ has the representation \eqref{eq:solf}
then 
\[
\max\limits_{0\leq t\leq 1}\|f(\cdot,t)\|_{L^2(0,1)}\leq \exp(\sqrt{2}\|v\|_{L^2(0,1,H^2(0,1))})(\|f_0\|_{L^2(0,1)}+\|z\|_{L^2((0,1)^2)}).
\]
\end{lemma}

\begin{proof} 
Let $D\Phi(\cdot,t)$ be the spatial derivative of $\Phi(\cdot,t)$. Notice that that for every $x\in [0,1]$
\[
\partial_t D(\Phi(x,t))=D (\partial_t \Phi(x,t))=D(v(\Phi(x,t),t))
=  v_x(\Phi(x,t),t)) D\Phi(x,t).
\]
Integrating in $t$, since $D\Phi(x,0))=1$
\begin{equation} \label{eq:DPhi}
D \Phi(x,t)=\exp  \left(\int_0^t v_x(\Phi(x,s),s))\,\dd s\right) \leq \exp(\sqrt{2t} \|v\|_{L^2(0,1,H^2(0,1))}),
\end{equation}
where the constant $\sqrt{2}$ comes form the embedding $H^1_0((0,1))\hookrightarrow L^{\infty}((0,1))$.
Therefore
\begin{equation}\label{eq:DXestimate}
\exp(-\sqrt{2t} \|v\|_{L^2(0,1,H^2(0,1))})  \leq D \Phi(x,t) \leq \exp(\sqrt{2t} \|v\|_{L^2(0,1,H^2(0,1))})
\end{equation}
since for every $t\in [0,1]$
\begin{equation}\label{est:H2C1}
\left|\int_0^t v_x(\Phi(\tacka,s),s))\,\dd s \right|\leq \int_0^t \| v_x(\Phi(\tacka,s),s))\|_{L^{\infty}(0,1)}\,\dd s\leq  \sqrt{2t} \|v\|_{L^2(0,1,H^2(0,1))}.
\end{equation}

 By Lemma \ref{lem:representation_sol}, for a.e. $y\in [0,1]$, $t \in [0,1]$,
\begin{equation}\label{representation_y}
f(y,t)=f_0(\Phi(\tacka, t)^{-1}(y))+\int_0^t z(\Phi(\Phi(\tacka, t)^{-1}(y),s),s)\,\dd s.
\end{equation}
By change of variables, for the first term on the right-hand side of \eqref{representation_y}
\begin{align*}
\int_0^1 f^2_0(\Phi(\tacka, t)^{-1}(y))\,\dd y & =\int_0^1 f^2_0(\Phi(\tacka, t)^{-1}(y)) D\Phi(\tacka, t)^{-1} D\Phi(\tacka, t)\,\dd y\\
& \leq  \|f_0\|^2_{L^2(0,1)} \exp(\sqrt{2}\|v\|_{L^2(0,1,H^2(0,1))}).
\end{align*}
For the second term on the right-hand side of \eqref{representation_y}, using estimate \eqref{eq:DXestimate} twice and applying Jensen's inequality, 
\[
\left\|\int_0^t z(\Phi(\Phi(\tacka, t)^{-1},s),s)\,\dd s \right\|_{L^2(0,1)}\leq \exp(\sqrt{2}\|v\|_{L^2(0,1, H^2(0,1))})\|z\|_{L^2((0,1)^2)}
\]
Plugging the above estimates into \eqref{representation_y} we obtain
\begin{align*}
\|f(\cdot,t)\|_{L^2}&\leq \|f_0\|_{L^2(0,1)} \exp \left(\frac{\sqrt{2}}{2}\|v\|_{L^2(0,1,H^2(0,1))}\right)+\exp(\sqrt{2}\|v\|_{L^2(0,1,H^2(0,1))} ) \, \|z\|_{L^2((0,1)^2)}\\
&\leq \exp(\sqrt{2}\|v\|_{L^2(0,1,H^2(0,1))})(\|f_0\|_{L^2(0,1)}+\|z\|_{L^2((0,1)^2)}).
\end{align*}
Therefore, $
\max\limits_{0\leq t\leq 1}\|f(\tacka,t)\|_{L^2(0,1)}\leq \exp \left(\sqrt{2}\|v\|_{L^2(0,1,H^2(0,1))})(\|f_0\|_{L^2(0,1)}+\|z\|_{L^2((0,1)^2)} \right).$

\end{proof}

\begin{remark}  \label{rem:z2}
We note that we can combine Remark \ref{rem:zformula} and formula \eqref{eq:DPhi} to get another representation of the weak solution of $\partial_t z + \partial_x(vz)$. Namely $z$ is characterized by 
\[ z(\Phi(x,t),t) \, D\Phi(x,t) = z(x,0) \]
for a.e. $x,t$.
\end{remark}

\begin{lemma}\label{solution_continuity}
Given $v\in\mathcal V$, $z\in L^2((0,1)^2)$. Let $f_0\in L^2(0,1)$. If $f\in L^{2}((0,1)^2)$ is a weak solution to the initial value problem
 \[
 f(\cdot,0)=f_0,\quad\partial_t f+\partial_x f\cdot v=z,
 \]
in the sense that for every $\phi \in C^{\infty}([0,1]^2)$, 
\begin{align}
\begin{split}\label{eq:weak-IVP}
-\int_0^1\int_0^1f\partial_t\phi & \,\dd x\,\dd t-\int_0^1 f_{0}(x)\phi(x,0)\,\dd x\\
&=\int_0^1\int_0^1  f\phi\partial_x v\,\dd x\,\dd t
+\int_0^1 \int_0^1f\partial_x \phi v\,\dd x\,\dd t
+\int_0^1 \int_0^1\phi z\,\dd x\,\dd t.
\end{split}
\end{align}
Then $f\in C(0,1,H^{-1})$ and
\begin{align*}
&\|\partial_t f\|_{L^2(0,1,H^{-1})} 
\leq \left[1 + \| v \|_{L^2(0,1, H^2)  }\exp(\sqrt{2}\|v\|_{L^2(0,1,H^2)}) \right](\|f_0\|_{L^2(0,1)}+\|z\|_{L^2((0,1)^2)}).
\end{align*}
\end{lemma} 
\begin{proof}
For any $\phi \in L^2((0,1), H^1(0,1))$
\begin{align*}
\left|\int_0^1 \int_0^1  (f\partial_x  v  + z ) \right. & \left. \phi \,\dd x \dd t + \int_0^1 \int_0^1   f v  \partial_x  \phi \,\dd x \dd t \right| \\
 \leq & \int_0^1 \|f(\cdot,t) \|_{L^2} \| \partial_x  v(\cdot,t)\|_{L^2}\| \phi (\cdot,t) \|_{L^\infty}+ \|f(\cdot,t) \|_{L^2} \|v(\cdot,t)\|_{L^\infty}\|\partial_x \phi(\cdot,t)\|_{L^2} \dd t \\
&+  \| z \|_{L^2((0,1)^2)}  \|\phi \|_{L^2((0,1)^2)}\\
 \leq & \left(\sqrt{2} \|f\|_{L^\infty(0,1, L^2)} \, \|v\|_{L^2(0,1, H^1)}  +   \| z \|_{L^2((0,1)^2)} \right) \|\phi\|_{L^2(0,1, H^1)},
\end{align*}
Thus if we define $\partial_t f(t)$ to be the mapping from $H^1$ to
$\R$ defined by 
\[ \langle \partial_t f(t), g \rangle = \int_0^1 \left(f(\tacka,t)  \partial_x  v(\tacka,t)  + z(\tacka,t)  \right) g \,\dd x + \int_0^1   f(\tacka, t)  v(\tacka,t)  \partial_x  g \,\dd x \]
we see that $\partial_t f \in L^2(0,1, H^{-1}(0,1))$ and its norm is controlled by the right hand side of the inequality above. 
Furthermore from the definition of weak solution \eqref{eq:weaksol} follows that $\partial_t f$ is indeed a weak derivative in time of $f$ in $H^{-1}$. Therefore $f \in C(0,1, H^{-1})$.

Therefore, combining Lemma \ref{f_n_energy_est} and the above estimates, we derive that
\begin{align}\label{L2bound: partial_tf}
\begin{split}
\|\partial_t f\|_{L^2(0,1,H^{-1})}& \leq  \| z \|_{L^2((0,1)^2)} + \sqrt{2}\| v \|_{L^2(0,1, H^2)  }\|f\|_{L^\infty(0,1, L^2)}\\
&\leq [1 + \sqrt{2} \| v \|_{L^2(0,1, H^2)  }\exp(\sqrt{2}\|v\|_{L^2(0,1,H^2)})]\cdot(\|f_0\|_{L^2(0,1)}+\|z\|_{L^2((0,1)^2)})
\end{split}
\end{align}
\end{proof}

Now we  prove Theorem \ref{th:minimizer_existence}.
\begin{proof}[Proof of Theorem \ref{th:minimizer_existence}]
To apply the direct method of calculus of variations to show the existence of a minimizing path,  we show compactness and lower-semicontinuity.

We recall that by claim (i) of Proposition \ref{prop:properties} we know that there exists a path of finite action. 
Thus for any minimizing sequence $(f_n,v_n,z_n)$, we have that $\{v_n\}_{n\in\mathbb N}$ is bounded in $L^2(0,1,H^2 \cap H^1_0)$ and $\{z_n\}_{n\in\mathbb N}$ is bounded  in $L^2((0,1)^2)$. Since $(L^2((0,1),H^s(0,1)))^*=L^2((0,1),H^{-s}(0,1))$ for any $s$ 
, Banach-Alaoglu Theorem allows us to extract a subsequence, still denoted by $(v_n,z_n)$, such that $v_n\rightharpoonup v$ and $z_n\rightharpoonup z$ weakly in $L^{2}(0,1,H^2(0,1) \cap H^1_0(0,1))$ and $L^{2}((0,1)^2)$, respectively. 

Since $H^2(0,1)\hookrightarrow C^1(0,1)$, $\,v_n\in L^2(0,1, C^1(0,1))$. 
 By Lemma \ref{lem:representation_sol}, $f_n$ has the representation \eqref{eq:solf} with $v,z,\Phi$ replaced by $v_n,z_n,\Phi_n$ respectively.
\smallskip

We claim that along a subsequence 
\begin{equation} \label{subseq_cpt}
f_n\to f \quad \; \text{in }\; L^2(0,1, H^{-\frac12}(0,1)).
\end{equation} 

First notice that by Lemma \ref{f_n_energy_est}, $\{f_n\}_{n\in\mathbb N}$ is a bounded set in $L^{\infty}(0,1, L^2(0,1))$. There exists $f\in L^{\infty}(0,1, L^2(0,1))$ such that up to a subsequence, $f_n\overset{*}{\rightharpoonup} f$ in $L^{\infty}(0,1, L^2(0,1))$. 
By relabeling we asume that the subsequence is the whole sequence. 
In addition, 
\[
\|f\|_{L^{\infty}(0,1, L^2(0,1))}\leq\liminf\limits_{n\to\infty}\|f_n\|_{L^{\infty}(0,1, L^2(0,1))}.
\]

Applying Lemma \ref{solution_continuity} to $(f_n,v_n,z_n)$, $\{\partial_t f_n\}_{n\in\mathbb N}$ is a bounded set in $L^1(0,1,H^{-1}(0,1))$. By Lions-Aubin lemma \cite{ChJuLi14}, since $\{f_n\}_{n\in\mathbb N}$ is a bounded set in $L^2(0,1,L^2(0,1))$, and $H^{-1}\hookrightarrow H^{-\frac12}\hookrightarrow L^2$ compactly, there exists $\tilde f$ in $L^2(0,1, H^{-\frac12}(0,1))$, such that up to a subsequence (still denoted by $f_n$) such that $f_n\to \tilde f$ in $L^2(0,1, H^{-\frac12}(0,1))$. 
In fact, $\tilde f=f$. Namely $f$ is the weak * limit of $f_n$ in $L^{\infty}(0,1,L^2(0,1))$, and thus the weak limit in $L^2(0,1,H^{-\frac12}(0,1))$.
 By the uniqueness of weak limit, $\tilde f=f$. 
 
\smallskip

To summarize: there exists $f\in L^2((0,1)^2), v\in \mathcal V$ and $z\in L^2((0,1)^2)$ such that
\begin{align*}
&f_n \overset{*}{\rightharpoonup} f \quad\text{in }L^{\infty}(0,1,L^2(0,1))\\
&f_n \to f \quad\text{in }L^2(0,1,H^{-\frac12}(0,1))\\
&z_n \rightharpoonup z \quad\text{in  }L^2(0,1,L^2(0,1))\\
&v_n \rightharpoonup v \quad\text{in }L^2(0,1; H^{2}(0,1)).
\end{align*}

Now pass to the limit in \eqref{eq:weaksol}. Since $\phi\in C^{\infty}([0,1]^2)$. First consider the term $\int_0^1\int_0^1  f_n\phi \partial_x v_n\,\dd x\,\dd t$. 
\begin{align*}
\left|\int_0^1\int_0^1  f_n\phi \partial_x v_n\,\dd x\,\dd t \right.&\left.   -\int_0^1\int_0^1  f\phi \partial_x v\,\dd x\,\dd t\right |\\
\leq\, & \left|\int_0^1\int_0^1  f_n\phi \partial_x v_n\,\dd x\,\dd t-\int_0^1\int_0^1  f\phi \partial_x v_n\,\dd x\,\dd t\right |\\
&+\left|\int_0^1\int_0^1  f\phi \partial_x v_n\,\dd x\,\dd t-\int_0^1\int_0^1  f\phi \partial_x v\,\dd x\,\dd t\right |\\
\leq\, & \int_0^1 \|f_n-f\|_{H^{-\frac12}(0,1)}\|\phi \partial_x v_n\|_{H^{\frac12}(0,1)}\,\dd t+\left|\int_0^1\int_0^1  f\phi( \partial_x v_n-\partial_x v)\,\dd x\,\dd t\right |\\
\leq\, & \|f_n-f\|_{L^2((0,1),H^{-\frac12}(0,1))}\sup\limits_{n\in\mathbb N}\|\phi\partial_x v_n\|_{L^2((0,1),H^{\frac12}(0,1))}\\
&+\left|\int_0^1\int_0^1  f\phi( \partial_x v_n-\partial_x v)\,\dd x\,\dd t\right |=:\,A_1+B_1,
\end{align*}
Consider $A_1$. For every $t$,  $\|\phi \partial_xv_n\|_{H^{\frac12}}\leq C\|\phi\|_{H^{1}}\|\partial_x v_n\|_{H^{\frac12}}\leq C \|\phi\|_{C^{1}}\|\partial_x v_n\|_{H^{\frac12}}$ by multiplication of Sobolev functions (see Theorem 7.4 of \cite{https://doi.org/10.48550/arxiv.1512.07379}). Then 
\begin{align*}
\sup\limits_{n\in\mathbb N}\|\phi\partial_x v_n\|_{L^2((0,1),H^{\frac12}(0,1))}&\leq C\max\limits_{t\in [0,1]}\|\phi(t)\|_{C^1}\sup\limits_{n\in\mathbb N}\|\partial_x v_n\|_{L^2((0,1),H^{\frac12}(0,1))}\\
&\leq C_{\phi} \sup\limits_{n\in\mathbb N}\|v_n\|_{L^2((0,1),H^{2}(0,1))},
\end{align*}
where the constant $C_{\phi}$ only depends on $\phi$. Combining with $f_n \to f$ in $L^2(0,1,H^{-\frac12}(0,1))$, we obtain that $A_1\to 0$ as $n\to\infty$. $B_1\to 0$ since $\phi f\in L^2(0,1,L^2(0,1))$ which could serve as a test function for $\partial_x v_n \rightharpoonup \partial_x v$ in $L^2(0,1, H^{1}(0,1))$. 
By analogous argument, \[\left|\int_0^1\int_0^1  f_n v_n\partial_x\phi \,\dd x\,\dd t  -\int_0^1\int_0^1  fv \partial_x \phi\,\dd x\,\dd t\right |\to 0.\]

For the other terms in \eqref{eq:weaksol}, passing to the limit is straightforward. Therefore $f$ satisfies \eqref{eq:weaksol}. By Lemma \ref{solution_continuity}, $f\in C(0,1, H^{-1}(0,1))$.

\smallskip

The lower semicontinuity of the action function follows directly from the lower-semicontinuity of norms with respect to the weak convergence. Therefore $(f,v,z)$ is a minimizer.

\end{proof}
\smallskip

Showing that $d_{HV}$ is a metric on $L^2(0,1)$ is straightforward. In particular the triangle inequality is obtained by concatenating minimizers with arc-length parameterization in which 
\[ d_{HV( \kappa, \lambda, \veps)} =
\int_0^1 \sqrt{\frac12  \int_0^1   \kappa v^2 + \lambda v_x^2  + \varepsilon v_{xx}^2 + z^2    \, \dd x } \;  \dd t  = \sqrt{ A_{\kappa, \lambda, \veps}(f,v,z)} . \]
The existence of such raparameterization follows from reparameterization result which states that any admissible path $(f,v,z)$ can be reparameterized in time by any absolutely continuous diffeomorphism of $[0,1].$
The lemma is a just variation of the lemma on rescaling of distributional solutions to continuity equations (see Lemma 8.1.3 of \cite{AGS}).
\begin{lemma}\label{cfv}
Let ${\rm t}:=s\in [0,1]\to {\rm t}(s)\in [0,1]$ be strictly increasing absolutely continuous map with absolutely continuous inverse ${\rm s}:={\rm t}^{-1}$. In addition, ${\rm t}(0)=0,{\rm t}(1)=1$. Then $(f(x,t),v(x,t),z(x,t))$ is a weak solution of 
\[
\partial_t f+\partial_x f\cdot v=z \quad\text{in}\quad [0,1]\times [0,1]
\]
with boundary condition $f(\cdot,0)=f_0, f(\cdot,1)=f_1$ if and only if \[\hat f(x,s)=:f(x,{\rm t}(s)),\quad \hat v(x,s)={\rm t}'(s)v(x,{\rm t}(s)),\quad\hat z(x,s)):={\rm t}'(s)z(x,{\rm t}(s))\]
is the weak solution of
solution of 
\[
\partial_s \hat f+\nabla \hat f\cdot \hat v=\hat z \quad\text{in}\quad [0,1]\times [0,1]
\]
with boundary condition $\hat f(\cdot,0)=f_0, \hat f(\cdot,1)=f_1$.
\end{lemma}


\begin{lemma}
The distance $d_{HV}$ is metric on $L^2(0,1)$.
\end{lemma}
\begin{proof}
We claim that the positivity follows from the existence of minimizers in Proposition \ref{th:minimizer_existence}. If $f_0=f_1$ in $L^2(0,1)$, it is obvious that $d_{HV}(f_0,f_1)=0$.
Consider the case $f_0\neq f_1$ in $L^2(0,1)$.
Observe that for $(f,v,z)\in\mathcal A$, we could reparameterize by arc length 
in time to make the quantity $\left(\int_0^1 \kappa v^2(x,s)+\lambda v^2_x(x,s)+ \varepsilon v^2_{xx}(x,s)+ z^2(x,s)\,\dd x \right)^{\frac12}$ constant in $s$,with the reparameterization denoted by $(\hat f,\hat v,\hat z)$. In particular, the reparameterization and its inverse are almost everywhere differentiable. Applying Lemma \ref{cfv} and changing variables, by Jensen's inequality we have
\[
A_{(\kappa,\lambda,\varepsilon)}(\hat f,\hat v,\hat z)\leq A_{(\kappa,\lambda,\varepsilon)}(f,v,z).
\]


According to Jensen's inequality, if $(f,v,z)$ is a minimizer of $A(f,v,z)$ then 
it is simultaneously the minimizer of $\int_0^1 \sqrt{\frac12  \int_0^1   \kappa v^2 + \lambda v_x^2  + \varepsilon v_{xx}^2 + z^2    \, \dd x } \;  \dd t$. The existence of minimizer implies that $d_{HV}(f_0,f_1)>0$, since otherwise the minimizer $f(x,t)$ is constant and $f_0\equiv f(\cdot,t)\equiv f_1$ which yields contradiction.

The symmetry is direct from the definition. The triangle inequality follows from path concatenation. Consider $f_0, f_1, f_2 \in L^2(0,1)$. Let $(f_{01},v_{01},z_{01})$ be the minimizing path between $f_0$ and $f_1$ and let $(f_{12},v_{12},z_{12})$ be the minimizing path between $f_1$ and $f_2$. Define
\[  \tilde f(t) = \begin{cases} f_{01}(2t) \quad & \te{for } t \in [0,\frac12] \\
f_{12}(2t-1)  & \te{for } t \in (\frac12,1].
\end{cases}\]
$\tilde v(t)$  and $\tilde z(t)$ are defined analogously. It immediately gives that $d_{HV}(f_0, f_2)\leq d_{HV}(f_0, f_1) + d_{HV}(f_1, f_2)$.
\end{proof}

\subsection{Completeness of the $d_{HV}$ metric space and the  characterization of its topology.} \label{sec:complete}
Thus far, we have shown that $d_{HV}$ is a metric on $L^2$. We now establish completeness.

\begin{proposition} \label{lem:complete}
The metric space $(L^2(0,1), d_{HV})$ is complete. 
\end{proposition}
\begin{proof}


Notice that $d_{HV( \kappa, \lambda, \veps)}\asymp d_{HV( 1, 1, 1)}$, it suffices to consider $ d_{HV( 1, 1, 1)}$.

 Let $\{g_n\}_{n\in\mathbb N} \subset L^{2}(0,1)$ be a Cauchy sequence in $d_{HV}$. Without a loss of generality, we assume that for all $n\in\mathbb N$, $d_{HV}(g_1,g_n)\leq 1$.
By Proposition \ref{th:minimizer_existence},  for $n>1$, let $(g_{1,n},v_{1,n},z_{1,n})$ denotes the minimizer of the action $A(f,v,z)$ with the admissible set $\mathcal A(g_1,g_n)$. Then Lemma \ref{f_n_energy_est} and the definition of $d_{HV}$ give that
\begin{align*}
\|g_n\|_{L^2(0,1)}&\leq \exp(\sqrt{2}\|v_{1,n}\|_{L^2(0,1,H^2(0,1))})(\|g_1\|_{L^2(0,1)}+\|z_{1,n}\|_{L^2((0,1)^2)})\\
&\leq \exp(\sqrt{2}d_{HV}(g_1,g_n))(\|g_1\|_{L^2(0,1)}+d_{HV}(g_1,g_n))\\
&\leq \exp(\sqrt{2})(\|g_1\|_{L^2(0,1)}+1).
\end{align*}
Then Banach-Alaoglu Theorem gives that there exists a subsequence $g_{n_k}$ converges to some $g_{\infty}\in L^2(0,1)$ weakly in $L^2(0,1)$. We denote the subsequence as $\{\tilde g_n\}_{n\in\mathbb N}$.

On the other hand, there exists further subsequences (which we relabel to be the original  subsequence) $\{v_{1,n}\}$ ,$\{z_{1,n}\}$ and $v_{1,\infty}, z_{1,\infty}$ such that 
\begin{align*}
&v_{1,n}\rightharpoonup v_{1,\infty},\quad \text{weakly in\:} L^2(0,1,H^2(0,1)),\\
&z_{1,n}\rightharpoonup z_{1,\infty},\quad \text{weakly in\:} L^2((0,1)^2).
\end{align*}

Meanwhile, by Lemma \ref{solution_continuity}, $\{\partial_t g_{1,n}\}_{n\in\mathbb N}$  is a bounded set in $L^2(0,1, H^{-1}(0,1))$. Then Lions-Aubin Lemma gives that up to a subsequence $g_{1,n}\to g_{1,\infty}$ in $L^2(0,1,H^{-\frac12}(0,1))$. Moreover we have $g_{1,n}\rightharpoonup g_{1,\infty}$ weakly * in $L^{\infty}(0,1,L^2(0,1))$.

Up to a subsequence,
\begin{align*}
 g_{n} & \rightharpoonup g_{\infty} \quad\text{in }L^{2}(0,1)\\
g_{1,n} & \overset{*}{\rightharpoonup} g_{1,\infty} \quad\text{in }L^{\infty}(0,1,L^2(0,1))\\
g_{1,n} & \to g_{1,\infty}  \quad\text{in }L^2(0,1,H^{-\frac12}(0,1))\\
z_{1,n} & \rightharpoonup z_{1,\infty}\quad\text{in  }L^2(0,1,L^2(0,1))\\
v_{1,n} & \rightharpoonup v_{1,\infty} \quad\text{in }L^2(0,1, H^{2}(0,1)).
\end{align*} 
By a similar argument as in the proof of Theorem  \ref{th:minimizer_existence}, we could verify that $(g_{1,\infty},v_{1,\infty},z_{1,\infty})$ is in the admissible set $\mathcal A(g_1,g_{\infty})$. Thus


\begin{align*}
d^2_{HV}(g_1,g_{\infty})
&\leq \liminf\limits_{n\to\infty}\|v_{1,n}\|^2_{L^2(0,1,H^2(0,1))}+ \liminf\limits_{n\to\infty}\|z_{1,n}\|^2_{L^2((0,1)^2)}\leq  \liminf\limits_{n\to\infty} d^2_{HV}(g_1,g_n).
\end{align*}
Then without loss of generality, we could take a further sequence in $\{\tilde g_n\}_{n\in\mathbb N}$ such that for every $n$, for any $m\geq n$, $d_{HV}(\tilde g_{m},\tilde g_n)<\frac1n$. We repeat the above arguments with $g_1$ replaced by $\tilde g_n$. Then we have
\[
d^2_{HV}(\tilde g_n,g_{\infty})\leq \liminf\limits_{m\to\infty} d^2_{HV}(\tilde g_n,\tilde g_m)<\frac{1}{n^2}.
\]
Now we have a subsequence  $\{\tilde g_n\}_{n\in\mathbb N}$ that converges to $g_{\infty}$ in $d_{HV}$. For any $\delta>0$, there exists $K_1\in\mathbb N$ such that $\frac{1}{K_1}<\frac{\delta}{2}$. And there exists $K_2>0$ such that for all $m_1,m_2>K_2$, $d_{HV(g_{m_1},g_{m_2})}<\frac{\delta}{2}$. Choose $N=\max\{K_1,K_2\}$,
\[
d_{HV}(g_n,g_{\infty})\leq d_{HV}(g_n,\tilde g_{N})+d_{HV}(\tilde g_N,g_{\infty})<\delta,
\]
which gives that $g_n\to g_{\infty}$ in $d_{HV}$.
\end{proof}
\medskip

Our next goal is to show that convergence in $d_{HV}$ implies convergence in $L^2$. Towards that goal we first prove the following estimate.
\begin{lemma}\label{ineq:f1g1}
For any $f_0,f_1\in L^2(0,1)$ and $g_0\in C^{\infty}$. Let $(f,v,z)$ be a minimizer of the action with the admissible set $\mathcal A(f_0,f_1)$ and  $\psi:=\Phi(\cdot,1)$, where $\Phi(x,t)$ is the flow map of $v(x,t)$. Define $g_1=g_0\comp\psi^{-1}$.

Then $g_1\in C^1$, satisfies $\int_0^1 |g_1'| \dd x  =\int_0^1 |g_0'| \dd x$ and
\begin{align}\label{ineq:fg_approx}
\|f_1-g_1\|_{L^2}
\leq \exp(\sqrt{2} d_{HV}(f_0,f_1))\, \left [\|f_0-g_0\|_{L^2}+d_{HV}(f_0,f_1)\right ].
\end{align}
\end{lemma}

\begin{proof}



$\psi^{-1}$ is a $C^1$ diffeomorphism 
immediately gives that $g_1\in C^1$ and $\int_0^1 |g_0'|\,\dd x=\int_0^1 |g_1'|\,\dd x$. Recall that by \eqref{eq:DXestimate}, we have 
\[
\left|D \Phi(\cdot,t) \right|\leq \exp(\sqrt{2t}\|v\|_{L^2(0,1,H^2(0,1))}),\quad \forall t\in [0,1].
\]
In particular, 
\begin{equation} \label{est:detDpsi}
|D \psi|\leq \exp(\sqrt{2}\|v\|_{L^2(0,1,H^2(0,1))}).
\end{equation}

By Lemma \ref{lem:representation_sol}, for a.e. $y\in [0,1]$,
\[
f_1(y)-g_1(y)= f_0(\psi^{-1}(y))+\int_0^1 z(\Phi(\psi^{-1}(y),s),s)\,\dd s-g_0(\psi^{-1}(y)).
\]
By an analogous estimate as in the proof of Lemma \ref{f_n_energy_est}, we have
\begin{align}\label{f1g1est_0}
\begin{split}
\|f_1-g_1\|_{L^2}&\leq \left\|\int_0^1 z(\Phi(\psi^{-1}(\tacka),s),s)\,\dd s \right\|_{L^2}+\|f_0\comp \psi^{-1}-g_0\comp \psi^{-1}\|_{L^2}\\
&\leq \exp(\sqrt{2}\|v\|_{L^2(0,1, H^2(0,1))})\|z\|_{L^2((0,1)^2)}+\exp\left(\frac{\sqrt{2}}{2}\|v\|_{L^2(0,1,H^2(0,1))} \right)\|f_0-g_0\|_{L^2}\\
&\leq \exp(\sqrt{2} d_{HV}(f_0,f_1))\left [ \|f_0-g_0\|_{L^2}+  d_{HV}(f_0,f_1)\right].
\end{split}
\end{align}



\end{proof}

\begin{theorem}\label{conv_L2}
Let $\{f_n\}_{n\in\mathbb N}\subseteq L^2(0,1)$. If $f_n\to f$ in $d_{HV}$, then $f_n\to f$ in $L^2$.
\end{theorem}
\begin{proof}
It suffices to show the conclusion $d_{HV(1,1,1)}$. We will prove this by showing that for any subsequence of $\{f_n\}_{n\in\mathbb N}$, there exists a further subsequence that converges to $f$ in $L^2$. 

Let $K:=2\exp(\sqrt{2})$. Let $0<\delta<1$ be arbitrarily chosen. By extracting a subsequence, we can assume that for every $n$, $d_{HV}(f_n,f)\leq \frac{\delta}{3K}<1$.
Since $C^{\infty}$ is dense in $L^2$, there exists $g\in C^{\infty}$ such that $\|f-g\|_{L^2}\leq \frac{\delta}{3K}$. For any $n$, let $\psi_n$ be the time-$1$ flow map of $v_n$, where $(f_{n},v_{n},z_{n})$ denotes the minimizer of the action $A(f,v,z)$ with the admissible set $\mathcal A(f,f_n)$.  Take $g_n:=g\comp \psi_n^{-1}$ if $f_n\neq f$, and take $g_n=g$ if $f_n=f$.

We will show that $\{g_n\}_{n\in\mathbb N}$ is a bounded set in $W^{1,2}(0,1)$. Thus without loss of generality, we assume $f_n\neq f$. By Lemma \ref{ineq:f1g1},
\begin{align}
\begin{split}
&\|f_n-g_n\|_{L^2}\leq  \exp(\sqrt{2} d_{HV}(f,f_n))\left [\|f-g\|_{L^2}+ d_{HV}(f,f_n)\right ]<\frac{\delta}{3}.
\end{split}
\end{align}
On the other hand, by change of variables, 
\[\|g_n\|^2_{L^2}\leq \|g\|^2_{L^2}\|D\psi_n\|_{L^{\infty}}\leq \|g\|^2_{L^2}\exp(\sqrt{2}\|v_n\|_{L^2(0,1,H^2(0,1))}),\]
where we use the estimate (\ref{est:detDpsi}) with $\psi,v$ replaced by $\psi_n, v_n$. Then we have
\[\|g_n\|_{L^2}\leq \|g\|_{L^2}\exp \left(\frac{\sqrt{2}}{2}d_{HV}(f,f_n) \right)\leq \|g\|_{L^2}\exp\left(\frac{\sqrt{2}}{2} \right),\] 
which implies that $\{g_n\}_{n\in\mathbb N}$ is a bounded set in $L^2$. Since $\psi_n^{-1}$ is the time-$1$ flow of $v_n(\cdot,1-t)$, by (\ref{eq:DPhi}),
\[
\int_0^1 |D g_n|^2\,\dd x\leq \exp(\sqrt{2}\|v_n\|_{L^2(0,1, H^2(0,1))})\|g'\|^2_{L^2}\leq \exp(\sqrt{2}d_{HV}(f,f_n))\|g'\|^2_{L^2}
\]
which implies that $\{D g_n\}_{n\in\mathbb N}$ is a bounded set in $L^2$. Thus $g_n\in W^{1,2}(0,1)$ and \[\sup\limits_n\|g_n\|_{W^{1,2}}< \exp(\sqrt{2})\left [\|g\|_{L^2}+\|g'\|^2_{L^2} \right ].\] Hence $\{g_n\}_{n\in\mathbb N}$ is a bounded set in $W^{1,2}(0,1)$. By Morrey's inequality and Arzela-Ascolli compactness criterion, there exists a subsequence $g_{n_k}$ that is Cauchy in $L^2$. For sufficiently large $N$, for any $i,j>N$, we have $\|g_{n_i}-g_{n_j}\|_{L^2}<\frac{\delta}{3}$. Thus
\begin{align*}
\|f_{n_i}-f_{n_j}\|_{L^2}&\leq \|f_{n_j}-g_{n_j}\|_{L^2}+\|g_{n_j}-g_{n_i}\|_{L^2}+\|f_{n_i}-g_{n_i}\|_{L^2}<\delta.
\end{align*}
This implies that up to a subsequence, $\{f_{n_k}\}_{k\in\mathbb N}$ is a Cauchy sequence in $L^2$. By completeness of $L^2$, it converges to some $\tilde f\in L^2$.  $\tilde f$ is also the weak limit of the subsequence in $L^2$, by the uniqueness of weak limit, we should have $\tilde f=f$ (recall that $f$ is the weak limit of a subsequence of $\{f_{n_k}\}_{k\in\mathbb N}$.)

Therefore, up to a subsequence, $f_n\to f$ in $L^2$. For any subsequence $f_n\to f$, it admits a further subsequence that converges to $f$. This proves that for the whole sequence, $f_n\to f$ in $L^2$.
\end{proof}

\section{Properties of geodesics.} \label{sec:further_properties}

\subsection{Euler--Lagrange Equations}\label{sec:EL}
 Assume $( f,  v,  z ) \in \mathcal A(f_0, f_1)$ is a minimizer of the action \eqref{eq:action1} over the admissible set \eqref{eq:A}. 
Below we describe the first-order optimality conditions, first for general paths, \eqref{eq:ELv1} and \eqref{eq:ELz1}, and then under assumption that $f_0,f_1 \in H^1(0,1)$, in which case all conditions are partial differential equations.
 
To find the first-order optimality conditions, in other words the Euler--Lagrange equations,
we first fix $f$ and perform the first variation in $v$ and $z$. 
Motivated by a similar reasoning as in Section \ref{sec:tangent}, we introduce the space,
\begin{equation} 
    R_T(f) = \left\{ u \in L^2(0,1, H^2(0,1) \cap H^1_0(0,1)) \:: uf\in L^2(0,1,H^1(0,1)) \right\}.
\end{equation}
Note that for $u \in R_T(f)$, for a.e. $t\in [0,1]$, $u f_x = (uf)_x - u_x f \in L^2((0,1)^2)$.

Now we characterize the optimality condition analogous to that in Lemma \ref{lem:first_var_gen}. Taking $u\in R_T(f)$ and $h=uf_x$, which belongs to $L^2((0,1)^2)$, we have that, for any $s \in \R$, $\,(f,v+su,z+sh)$ satisfies the equation \eqref{eq:weaksol}, and thus $(f,v+su,z+sh) \in \mathcal A(f_0, f_1)$ . First variation of action \eqref{eq:action1} gives that $(f,v,z)$ satisfies:
\begin{equation}\label{eq:ELv1}
\forall u\in R_T(f)\qquad\int_0^1\int_0^1\kappa vu+\lambda v_x u_x+\veps v_{xx}u_{xx}+ z(uf_x)\,\dd x\dd t=0. \qquad
\end{equation}

We then turn to the optimality of $z$.   To carry out the argument let us denote that minimizer of the  action considered by $(\bar f, \bar v, \bar z)$.
Then $(\bar f, \bar z)$ is a critical point of $\tilde A(z) = \int_0^1 \int_0^1 z^2 \dd x dt$ over the set of  $(f,v,z) \in \mathcal A(f_0, f_1)$ such that $v = \bar v$.
This is a convex functional over linear constraint. Thus $(\bar f, \bar z)$ is a global minimizer for fixed $\bar v$. 

Furthermore by Lemma \ref{lem:representation_sol}, the constraint that $(f,v,z) \in \mathcal A(f_0, f_1)$, with $v= \bar v$ can be expressed as follows: for a.e. $x,t$
\begin{equation} \label{eq:fz2}
 f(\Phi(x,t),t)=f_0(x)+\int_0^t z(\Phi(x,s),s)\,\dd s.  
\end{equation}
In other words $z$ needs to satisfy that for a.e. $x \in (0,1)$
\begin{equation} \label{eq:z2}
f_1(\Phi(x,1))-f_0(x) = \int_0^1 z(\Phi(x,t),t)\,\dd t
\end{equation}
while $f(\tacka,t)$ for $0<t<1$ is defined by \eqref{eq:fz2}.
Since $f$ does not enter the action directly, we  minimize over $z$ alone and define $f$ by \eqref{eq:fz2}. After the change of variables
$\tilde{\mathcal A}(z) = \int_0^1 \int_0^1 z^2(\Phi(x,t),t) D\Phi(x,t) \dd x \dd t$. Hence if we define $w(x,t) = z(\Phi(x,t),t)$ the problem is transformed to 
\begin{align*} 
& \te{minimize} \int_0^1 \int_0^1 w^2(x,t) D\Phi(x,t) \dd x \dd t  \\
& \te{under constraint:} \,
f_1(\Phi(x,1))-f_0(x) = \int_0^1 w(x,t)\,\dd t \quad \te{ for a.e. }  x \in (0,1).
\end{align*}
By Cauchy-Schwarz inequality, the minimizer $w$ of this problem should satisfy that  $w(x,t) D\Phi(x,t)$ is independent of time. Thus $z$ which minimizes $\tilde A$ satisfies that for a.e. $x$
\[ z(\Phi(x,t),t) D\Phi(x,t) =: z(x,0). \]
By Remark \ref{rem:z2}, this implies that $z$ is a weak solution of 
\begin{equation} \label{eq:ELz1}
   \partial_t z + \partial_x ( zv) =0 \qquad \te{on } (0,1)^2,
\end{equation}
which is the desired first-order condition on $z$. 
\bigskip

If $f_0, f_1 \in H^1(0,1)$ then we establish in the  Proposition \ref{prop:f_H1} below that the minimizing path $f \in L^2(0,1,H^1(0,1))$. In that case we the resulting 
  Euler--Lagrange equations  satisfied by $f,v,z$, combined with the condition for belonging to $\mathcal A(f_0, f_1)$ can be expressed as follows:
\begin{align}
\veps v_{xxxx} -\lambda v_{xx} + \kappa v + z f_x & = 0 \te{ weakly on } (0,1)^2 \label{eq:elv}\\
 v = 0   \; \te{ and } \;  v_{xx} & = 0 \,\,\te{ on } \partial \Omega \times [0,1] \;  \label{eq:v-BCs}\\
z_t + (zv)_x & = 0  \te{ weakly on } (0,1)^2 \label{eq:elz}  \\
 f_t + f_x  v - z & = 0 \te{ weakly on } (0,1)^2 \label{eq:elf}\\
 f(\tacka,0)  =f_0, \; \te{ and }  \; f(\tacka,1)&=f_1. \label{eq:f-ICs}
\end{align}
The first equations, and boundary conditions,  follow directly from \eqref{eq:ELv1}. The other conditions are identical as before.

\subsection{Regularity of geodesics.}
This subsection presents results regarding the regularity of the geodesics.

\begin{proposition}\label{prop:f_H1}
 Assume $f_0,f_1\in H^1$ and let $(f,v,z) \in \mathcal A(f_0,f_1)$ be an action minimizing path. Then $f\in L^{\infty}(0,1,H^1(0,1))$, with quantitative estimates on the norm \eqref{H1_f_phi}, \eqref{H1_f_phi2}.
\end{proposition}

\begin{proof}
We remark that the existence of action minimizing $(f,v,z)\in \mathcal A(f_0,f_1)$ is guaranteed by Theorem \ref{th:minimizer_existence}. 
The argument in Section \ref{sec:EL} implies that $z$ solves $z_t+(zv)_x=0$, weakly.
As in Remark \ref{rem:zformula}, $z$ satisfies the formula \eqref{eq:zJ} and $f$ satisfies \eqref{eq:f-formula}.
By Theorem 2.2.2 in \cite{Ziemer1989WeaklyDF}, since $\;\Phi(\tacka,t)$ is bi-Lipschitz in $x$ it suffices to check if $f(\Phi(x,t),t)\in H^1$. We use the composition with $\Phi$ in several instances below. 

Let $b:=e^{\sqrt{2} \|v\|_{L^2(0,1,H^2)}}$.  Recall from \eqref{eq:DXestimate} that for all $x \in [0,1]$ and $t \in [0,1]$
\[ \frac{1}{b} \leq |D\Phi(x,t) | \leq b. \]


Note that for all $t$, 
$\; \|\int_0^tv_x ( \Phi(\tacka, s), s)\,\dd s \|_{L^{\infty}}\leq \|v_x\|_{L^1(0,1,L^{\infty})}\leq  \sqrt{2} \|v\|_{L^2(0,1,H^2)}.$
Thus for all $(x,t)$
\begin{equation} \label{eq:Jbelowabove}
   \frac{1}{b}  \leq   J(x,t) \leq e^{ \sqrt{2} \|v\|_{L^2(0,1,H^2)}}=b. 
\end{equation}

Change of variables provides that, for a.e. $t$,
\[\int_0^1 (\partial_x v_x(\Phi(x,t),t))^2 \dd x \leq \|D\Phi(\tacka,t)\|_{L^\infty} \| v_{xx}(t)\|_{L^2((0,1))}^2 \leq b \|v\|_{L^2(0,1,H^2)}.\]
Integrating in $t$ and Cauchy-Schwarz inequality imply that for all $t \in [0,1]$
\[\left\| \partial_x \int_0^t v_x ( \Phi(x, s), s) \dd s \right\|_{L^2(0,1)}^2 \leq 
t b\,  \|v\|^2_{L^2(0,1,H^2)}.\]
Combining with the  chain rule and the $L^\infty$ estimate on $J$ implies that 
$J \in L^\infty(0,1,H^1(0,1))$ and
\begin{equation} \label{eq:intJH1}
\forall t \in [0,1] \qquad \quad \left\|\int_0^t J(x,s)\,\dd s\right\|_{H^1(0,1)} \! \leq b+ b^{3/2} \,  \|v\|_{L^2(0,1,H^2)} \qquad \phantom{.}
\end{equation}


Since $J \geq \frac{1}{b}$,  $1/\int_0^1 J(x,t) dt$ is in $H^1(0,1)$. By chain rule \cite{Evans},
\begin{equation} \label{eq:1overJ}
\left\|\partial_x \left(1/\int_0^1 J(x,t) dt\right)\right\|_{L^2}\leq b^2\left\|\partial_x \left(\int_0^1 J(x,t) dt\right)\right\|_{L^2}\leq b^{\frac72}\|v\|_{L^2(0,1,H^2)}.
\end{equation}
Moreover by change of variables and estimate \eqref{eq:DXestimate},
$\| \partial_x f_1 ( \Phi(x, 1) )\|_{L^2}\leq\sqrt{b}\|\partial_x f_1\|_{L^2}.$ 
From \eqref{eq:zJ}, \eqref{eq:Jbelowabove}, \eqref{eq:intJH1}, and \eqref{eq:1overJ}, via the product rule and using that $\| gh\|_{L^2} \leq \|g\|_{L^2} \|h\|_{L^\infty}$, we obtain
\begin{align*}
\| \partial_x z(\Phi(x, \tacka),\tacka) \|_{L^2((0,1)^2)} & \leq (\| \partial_x f_1(\Phi(\tacka,1)) \|_{L^2} + \| \partial_x f_0 \|_{L^2}) b^2 \\
& \phantom{\leq\:} + (\| f_1(\Phi(\tacka,1)) \|_{L^\infty}+ \|f_0 \|_{L^ \infty}) (b^\frac52 + b^\frac92) \|v\|_{L^2(0,1,H^2)}\\
& \leq 3(  \| f_1\|_{H^1} + \| f_0 \|_{H^1})(b^\frac52 + b^\frac92 \|v\|_{L^2(0,1,H^2)}).
\end{align*}



Note also that, via change of variables, 
\[ \int_0^1 \int_0^1 z^2(\Phi(x,t),t) \dd x \dd t \leq b \|z \|_{L^2((0,1)^2)}^2 \leq b \|f_1 - f_0\|_{L^2}^2. \]

From \eqref{eq:f-formula}, by a change of variables and estimate \eqref{eq:DXestimate}, for every $t\in [0,1]$,
\begin{equation}\label{H1_f_phi}
\|\partial_x f(\tacka,t)\|_{L^2(0,1)}\leq \sqrt{b} \|\partial_x f(\Phi(\tacka,t),t)\|_{L^2} \leq 6( \| f_1 \|_{H^1} + \|f_0\|_{H^1}) (1+b^3 +b^5 \|v\|_{L^2(0,1,H^2)}),
\end{equation}
which using that $\|v\|_{L^2(0,1,H^2)} \leq C_\veps  \|f_1 -f_0\|_{L^2}$ can be turned in an estimate where right-hand side depends only on initial and final signal. 

The estimate on $\|f(\tacka, t)\|_{L^2}$ from Lemma \ref{f_n_energy_est} and the estimate on $z$ above provide that 
\begin{equation} \label{H1_f_phi2}
 \|f(\tacka,t)\|_{L^2(0,1)} \leq b(\|f_0\|_{L^2} + \sqrt{b} \|f_0-f_1\|_{L^2} ), 
\end{equation}
which completes the proof. 
\end{proof}

\begin{remark} \label{rem:fsmooth}
The regularity above can be improved to spaces with more regularity. Here we now outline  argument that if $f_0$ and $f_1$ ate smooth then $f(\tacka,t)$ is smooth for all $t \in [0,1]$. 
For $f_0,f_1\in H^1(0,1)$ Proposition \ref{prop:f_H1} gives that the minimizing path $f\in L^{\infty}(0,1,H^1(0,1))$ and $z\in L^{\infty}(0,1,H^1(0,1))$. Thus in equation \eqref{eq:elv},  $zf_x\in L^{2}(0,1)$ for every $t\in [0,1]$. The results on elliptic boundary value problems imply $v(\tacka,t)\in H^4(0,1)$. Moreover, there exists a constant $C$ such that
\[
\|v(\tacka,t)\|_{H^4(0,1)}\leq C(\|zf_x(\tacka,t)\|_{L^2}+\|v(\tacka,t)\|_{H^2}).
\]
Thus \[\|v\|_{L^2(0,1,H^4)}\leq C(\|z\|_{L^2(0,1,H^1)}\|f\|_{L^2(0,1,H^1)}+\|v\|_{L^2(0,1,H^2)}), \]
where the right-hand side, by Proposition \ref{prop:f_H1}, is a function only depending on $f_0$ and $f_1$.

According the differentiability of ODE solution with respect to parameters, we obtain that for fixed $t$, the flow map $\Phi(x,t)\in C^3(0,1)$. In addition, $\Phi(x,t)\in L^{\infty}(0,1,C^3)$. Indeed, $D^2\Phi(x,t)$ satisfies the equation:
\[
\partial_t D^2\Phi(x,t)=( v_x(\Phi(x,t),t))^2 (D\Phi(x,t))^2+v_x(\Phi(x,t),t)D^2\Phi(x,t).
\]
By Gronwall's inequality and \eqref{eq:DXestimate} we obtain that $D^2\Phi\in L^{\infty}((0,1)^2)$. Inductively, $\Phi \in L^{\infty}(0,1,C^3)$.

Now we assume that $f_0,f_1\in H^4(0,1)$. The above discussion gives that $v\in L^2(0,1,H^4)$ and $\Phi\in L^{\infty}(0,1,C^3)$. Applying chain rule to \eqref{eq:zJ} and following a similar argument as in the proof of Proposition \ref{prop:f_H1}, we have that $z\in L^{\infty}(0,1, H^3(0,1))$. Then by \eqref{eq:f-formula}, $f\in L^{\infty}(0,1,H^3(0,1))$, with $\|f\|_{L^{\infty}(0,1,H^3(0,1))}$  bounded by a function of $f_0$ and $f_1$.

By an iterative argument, we conclude that for any $k\in\mathbb N$, $f,z\in L^{\infty}(0,1,H^k(0,1))$ and $v\in L^2(0,1,H^k(0,1))$, which implies the smoothness of $f(\tacka, t)$. 
\end{remark}
\medskip

\subsection{Stability of geodesics.}
Next, we study the stability of the geodesic in terms of a subsequence approximation.
\begin{proposition} \label{prop:stability_geod}
Assume $f_0,f_1 \in L^2(0,1)$,  $f^n_0, f^n_1 \in L^2(0,1)$ for all $n \in \N$, and 
$f^n_0 \to f_0, \;\; f^n_1 \to f_1$ in $L^2(0,1)$ as $n \to \infty$. 
Let $(f^n, v^n, z^n) \in \mathcal A(f^n_0, f^n_1) $ be action minimizing paths. Then there exists $(f,v,z) \in \mathcal A(f_0, f_1)$ such that along a subsequence
\begin{align*}
&f^n \overset{*}{\rightharpoonup} f \quad\text{in }L^{\infty}((0,1),L^2(0,1))\\
&f^n \to f \quad\text{in }C((0,1),(L^2(0,1), d_{HV}))\\
&z^n \rightharpoonup z \quad\text{in  }L^2((0,1),L^2(0,1))\\
&v^n \rightharpoonup v \quad\text{in }L^2([0,1]; H^{2}(0,1)).
\end{align*}
Furthermore $(f,v,z)$ is an action minimizing path between $f_0$ and $f_1$.
\end{proposition}
\begin{proof}
Since 
\begin{equation} \label{eq:HVtemp}
    d_{HV}(f^n_0, f^n_1) \leq \|f^n_0 - f_0\|_{L^2 } + d_{HV}(f_0, f_1) + \|f^n_1 - f_1\|_{L^2 },
\end{equation}
we conclude  that $A(f^n, v^n, z^n) \leq d_{HV}^2(f_0, f_1) +1 \leq \|f_0 - f_1\|_{L^2}^2 +1 $ for all $n$ large enough. We can assume, without loss of generality that the inequalities hold for all $n$. 
The proof of weak convergence mirrors the proof of Theorem \ref{th:minimizer_existence}. To show that $(f,v,z)$ is an action minimizing path, we note that by lower-semicontinuity $A(f,v,z) \leq \liminf_{n \to \infty} A(f^n,v^n,z^n)$. Since $A(f^n,v^n,z^n) = d_{HV}^2(f^n_0, f^n_1) $, from \eqref{eq:HVtemp} follows that $A(f,v,z)  \leq  d_{HV}^2(f_0, f_1) $. Thus 
 $A(f,v,z) =  d_{HV}^2(f_0, f_1) $.

We claim that for all $t$, $f^n(\tacka, t)$ is precompact in $L^2$. We note that by Lemma \ref{f_n_energy_est} the sequence
 $f^n(\tacka,t)$ is uniformly bounded in $L^2$.  To show precompactness in $L^2$, by \cite{FonLeo07}[Theorem 2.88] we need to show that for all $\delta>0$ there exists $h_0>0$ such that  for all $h \in [0, h_0]$ and all $n$
\begin{equation} \label{eq:unif_int}
 \|f^n(\tacka+h,t) - f^n(\tacka,t) \|_{L^2} < \delta. 
 \end{equation}

We proceed with the proof assuming \eqref{eq:unif_int}, and verity the condition at the end.

Observe that $t \mapsto f^n(\tacka,t)$ is a constant speed curve in $(L^2(0,1), d_{HV})$, since it is a geodesic. Therefore that family of functions $f^n$ considered as functions between $[0,1]$ and $(L^2(0,1), d_{HV})$ are uniformly Lipschitz and thus equicontinuous. Since we know that $\{ f^n(\tacka,t) \}_{n \in \N}$ is precompact in $(L^2(0,1), \| \tacka \|_{L^2})$, it is precompact in $L^2$ with respect to the $d_{HV}$ metric. Therefore by the Arzela-Ascoli Theorem, for metric-space valued functions, we get that $f^n$ is precompact in $C(0,1, (L^2(0,1), d_{HV})$. We note that this also implies that up to a subequence, $f^n(\tacka, t)$ converges to $f(\tacka, t)$ in $L^2(0,1)$ for all $t$ fixed. 
\smallskip

We now turn to proving that \eqref{eq:unif_int} holds. 
 By \eqref{eq:DXestimate}, for each $t$, $\Phi_t:=\Phi(\tacka,t):$ and $\Phi_t^{-1}$ are Lipschitz and the Lipschitz constant for both maps and for all $t \in [0,1]$ is bounded by $b_n:=  \exp(\sqrt{2}\|v^n\|_{L^2(0,1,H^2)})$. 

Next we show that the function $\eta^n(x,t):=\frac{\int_0^t J^n(x,s)\,\dd s}{\int_0^1 J^n(x,\tau)\,\dd \tau}$, where $J^n(x,t) = e^{-\int_0^t v^n_x ( \Phi^n(x, s), s) \dd s}$ is H\"older continuous in $x$.
Since $v_x^n$ is H\"older continuous, we have that for $x, y \in [0,1]$,

\[
|J^n(x,t)-J^n(y,t)|\leq e^{b_n}\int_0^t \|v_{xx}^n(\tacka,s)\|_{L^2} (\Lip(\Phi^n_s))^{\frac12}\,\dd s |x-y|^{\frac12}\leq  e^{b_n} \, b_n^{\frac12}\|v^n\|_{L^2(0,1,H^2)}|x-y|^{\frac12},
\]
since $|v_x(y) - v_x(y)| \leq |x-y|^\frac12 \|v_{xx}\|_{L^2}$.
This implies that for any $0\leq t\leq 1$, $\int_0^t J^n(x,s)\,\dd s$ is in $C^{0,\frac12}(0,1)$ with H\"older constant $ e^{b_n}\,b_n^{\frac12}\|v^n\|_{L^2(0,1,H^2)}$.  Observe that by \eqref{eq:Jbelowabove}, 
$\;b_n^{-1} \leq J^n(x,t)\leq b_n$.
Thus $(\int_0^1 J(x,\tau)\,\dd \tau)^{-1}$ is H\"older continuous with constant $ e^{b_n}\,b_n^{\frac52}\|v^n\|_{L^2(0,1,H^2)}$. Therefore for any $x,y$, for any $t$,
\[
|\eta^n(x,t)-\eta^n(y,t)|\leq (b_n e^{b_n}\,b_n^{\frac12}
+ b_n e^{b_n}\,b_n^{\frac52}) \, \|v^n\|_{L^2(0,1,H^2)}|x-y|^{\frac12}
\leq 2 e^{b_n}\,b_n^{\frac72}\|v^n\|_{L^2(0,1,H^2)}|x-y|^{\frac12}.
\]
By change of variables
\[
f^n(y,t)=\eta^n((\Phi^n_t)^{-1}(y),t)f_0^n((\Phi^n_t)^{-1}(y))+(1-\eta^n(\Phi_t^{-1}(y),t))f_1^n(\Phi_1\Phi_t^{-1}(y))
\]
We note that  $b:=\sup \{ b_n \::\: n \in \N\}$ is finite since $\| v^n \|_{L^2(0,1,H^2)}^2 \leq c \|f_0-f_1\|_{L^2}^2+1$ by the observation above, for some $c>0$ depending on $\lambda, \kappa$, and $\veps$. 

We extend $f^n$, $f_0$ and $f_1$ by zero to $\R$. We also extend $\Phi^n_t$ as identity ($\Phi^n_t(x)=x$) outside of $[0,1]$. 
 Choose $g_0,g_1\in C_c^{\infty}(\R)$ such that $\|f_0-g_0\|_{L^2(\R)}<\frac{\delta}{100b^3}$ and $\|f_1-g_1\|_{L^2(\R)}<\frac{\delta}{100b^3}$. Choose $N(\delta)$ such that, then for any $n>N(\delta)$, $\|f^n_0-f_0\|_{L^2}<\frac{\delta}{100b^3}$, $\|f^n_1-f_1\|_{L^2}<\frac{\delta}{100b^3}$. Then for $n>N(\delta)$, 
\begin{align*}
\|f_0^n((\Phi^n_t)^{-1}(y+h))-& f_0^n((\Phi^n_t)^{-1}(y))\|_{L^2(\R)}\\
\leq & \|f_0^n((\Phi^n_t)^{-1}(y+h))-g_0((\Phi^n_t)^{-1}(y+h))\|_{L^2(\R)}\\
&+\|g_0((\Phi^n_t)^{-1}(y+h))-g_0((\Phi^n_t)^{-1}(y))\|_{L^2(\R)}\\
&+\|g_0((\Phi^n_t)^{-1}(y))-f_0^n((\Phi^n_t)^{-1}(y))\|_{L^2(\R)}\\
\leq & 2\|D(\Phi^n_t)^{-1}\|^{\frac12}_{L^{\infty}}\|f^n_0-g_0\|_{L^2}+ \|g_0((\Phi^n_t)^{-1}(y+h))-g_0((\Phi^n_t)^{-1}(y))\|_{L^2(\R)}\\
\leq &\frac{\delta}{25b^2}+b \Lip(g_0)h,
\end{align*}
Similarly,
\[  \|f^n_1(\Phi^n_1(\Phi^n_t)^{-1}(y+h))-f^n_1(\Phi^n_1(\Phi_t^n)^{-1}(y))\|_{L^2(\R)} \leq \frac{\delta}{25b^2}+ b^2\Lip(g_1)h. \]
Moreover
\begin{equation}
\|\eta^n((\Phi^n_t)^{-1}(y+h),t)-\eta^n((\Phi^n_t)^{-1}(y),t)\|_{L^{\infty}(0,1)}\leq2 e^{b_n}\,b_n^{4}\|v^n\|_{L^2(0,1,H^2)}h^{\frac12}\leq C  h^{\frac12},
\end{equation}
where $C:=2\exp(b)b^{4}(\sqrt{c}\|f_1-f_0\|_{L^2}+1)$.

Therefore,
\begin{align*}
\|f^n&(\tacka,t) -f^n(\tacka+h,t)\|_{L^2(\R)}\\
\leq &\|\eta^n((\Phi^n_t)^{-1}(\tacka+h),t)f_0^n((\Phi^n_t)^{-1}(\tacka+h))-\eta^n((\Phi^n_t)^{-1}(\tacka),t)f_0^n((\Phi^n_t)^{-1}(\tacka))\|_{L^2}\\
&+\|(1-\eta^n(\Phi_t^{-1}(\tacka+h),t))f_1^n(\Phi_1\Phi_t^{-1}(\tacka+h))-(1-\eta^n(\Phi_t^{-1}(\tacka),t))f_1^n(\Phi_1\Phi_t^{-1}(\tacka))\|_{L^2}\\
\leq& \|\eta^n\|_{L^{\infty}}\|f^n_0(\Phi^n_t)^{-1}(\tacka+h)-f^n_0(\Phi_t^n)^{-1}(\tacka)\|_{L^2}\\
&+ \|\eta^n((\Phi^n_t)^{-1}(\tacka+h),t)-\eta^n((\Phi^n_t)^{-1}(\tacka),t)\|_{L^{\infty}}\|f_0^n\comp (\Phi_t^n)^{-1}\|_{L^2}\\
&+\|1-\eta^n\|_{L^{\infty}}\|f^n_1(\Phi^n_1(\Phi^n_t)^{-1}(\tacka+h))-f^n_1(\Phi^n_1(\Phi_t^n)^{-1}(\tacka))\|_{L^2}\\
&+\|\eta^n((\Phi^n_t)^{-1}(\tacka+h),t)-\eta^n((\Phi^n_t)^{-1}(\tacka),t)\|_{L^{\infty}}\|f_1^n\comp\Phi^n_1\comp (\Phi_t^n)^{-1}\|_{L^2}\\
\leq &b^2 \left(\frac{\delta}{25b^2}+b\Lip(g_0)h \right)+C h^{\frac12}b\|f_0^n\|_{L^2}+ (1+b^2)\left(\frac{\delta}{25b^2}+b^2\Lip(g_1)h\right)+C h^{\frac12}b\|f^n_1\|_{L^2}\\
\leq &\frac{3\delta}{25}+2b^3(\Lip(g_0)+\Lip(g_1))h+C b (2+\|f_0\|_{L^2}+\|f_1\|_{L^2})h^{\frac12}.
\end{align*}
We  pick $h_0$, which depend on $f_0,f_1,g_0,g_1$, such that for all $t$, for all $h\in [0,h_0]$, and $n>N(\delta)$, \eqref{eq:unif_int} holds.
For $n\leq N(\delta)$, there exists $h_n$ such that for all $h\in [0,h_n]$, \eqref{eq:unif_int} holds for $f^n$. Take $\tilde h:=\min\{h_0,h_1,\dots,h_{N(\delta)}\}$. For all $t$, for all $h\in [0,\tilde h]$, and for all $n$, \eqref{eq:unif_int} holds.
\end{proof}

\section{Numerical Scheme}\label{sec:scheme}

We propose an iterative minimization scheme to find the minimizers of the action \eqref{eq:obj} over the admissible paths  \eqref{eq:adm}. We first present two convex sub-problems on the continuous level by fixing $v$ and $f$, respectively. By solving each of the sub-problems, the action functional decays monotonically. We then present our discretization scheme of finding the optimal path based on this problem splitting.

\subsection{Two sub-problems}
Recall that we are interested in minimizing 
\begin{align}\label{eq:obj}
 A_{\kappa, \lambda, \veps}(f,v,z) = \frac12 \int_0^1 \int_0^1   \kappa v^2 + \lambda v_x^2  + \varepsilon v_{xx}^2 + z^2    \, \dd x \dd t,  
\end{align} 
over the set of admissible paths defined by 
\begin{equation}\label{eq:adm}
\mathcal{A} = \{ (f,v,z) : f_t = -f_x v + z, \, v(0,\cdot) = v(1,\cdot) = 0,  \,f(\cdot,0) = f_0,\, f(\cdot, 1) = f_1 \}.
\end{equation}
The Euler--Lagrange equations for this variational problem are given in~\eqref{eq:elv}-\eqref{eq:f-ICs} under the assumption that $f_0,f_1 \in H^1((0,1))$. The numerical methods presented in this section and the experiments in Section~\ref{sec:experiments} are based on the assumptions that the signals are $C^2((0,1))$. Due to the
regularity results of Proposition \ref{prop:f_H1} and Remark \ref{rem:fsmooth} as well as the stability result proved in Proposition~\ref{prop:stability_geod}, we expect that the geodesics for regular signals can be used to approximate those for general signals.

We point out that among $(f,v,z)$, the three variables we optimize over, $z$ is determined by $f$ and $v$ due to the constraint set~\eqref{eq:adm}. However, it is still impractical to directly minimize~\eqref{eq:obj} over~\eqref{eq:adm} due to the nonlinear constraint. Next, we translate the optimization problem to a fixed-point problem by working with the system of Euler--Lagrange equations~\eqref{eq:elv}-\eqref{eq:f-ICs}. 

More specifically, we find a solution to the Euler--Lagrange equations through two convex optimization sub-problems. The method alternates between fixing $v$ while finding the optimal $(f,z)$ and fixing $f$ while searching the optimal $(v,z)$, both for problem~\eqref{eq:obj}-\eqref{eq:adm}. This method shares similar flavors with many existing optimization algorithms. First, it is related to the so-called block coordinate descent method~\cite{tseng2001convergence} since we alternatingly update $(f,z)$ and $(v,z)$, the coordinate blocks in our problem. Based on~\cite[P.~266]{luenberger2015linear}, our algorithm also has local convergence since the action function has a unique minimum in each coordinate block. The fact that our updated new $(f,z)$ or $(v,z)$ is the exact minimizer for each of the sub-problem, sharing similar features with ADMM~\cite{wang2019global}. The corresponding optimal $(f,z)$ or $(v,z)$ are weighted projections onto the linear constraints determined by \eqref{eq:elv} - \eqref{eq:f-ICs}
(with the other variable fixed). 

Lastly, our method shares the same spirit with the so-called sequential quadratic programming (SQP)~\cite{boggs1995sequential}. For quadratic programming with nonlinear constraints, SQP solves a sequence of optimization sub-problems using a linearization of the constraints. In our method, we achieve linearization of~\eqref{eq:adm} by fixing $v$ or $f$.  Next, we will discuss in detail the two important sub-problems.

\subsubsection{From $v$ to $(f,z)$}\label{subsec:v2fz}


First, for a given $v$, we consider the sub-problem of finding $(f,z)$ which minimize the action \eqref{eq:obj}
under the constraint that $(f,v,z) \in \mathcal A$ for the $v$ fixed. 
We note that this reduces to minimizing 
a convex (in fact, quadratic) objective functional under a linear constraint:
\begin{equation}\label{eq:sub2}
\min_{f,z} \frac12 \int_0^1 \int_0^1  z^2  \, \dd x \dd t,  \quad \text{s.t.}\quad (f,v,z) \in \mathcal{A}.
\end{equation}
The first-order optimality conditions are given by equations \eqref{eq:elz}, \eqref{eq:elf} and~\eqref{eq:f-ICs}:
\begin{align}\label{eq:sub1}
\begin{split}
z_t + (zv)_x & = 0,  \\
f_t + f_x  v - z & = 0,  \\
f(\tacka,0) =f_0,\,
\,& f(\tacka,1)=f_1. 
\end{split}
\end{align}

To solve~\eqref{eq:sub1}, we can use the Lagrangian approach. First, we can obtain a flow map $\Phi(x,t)$ solving~\eqref{eq:flow-map}, and an analytical solution for $z(\Phi(x,t) , t)$ 
presented in \eqref{eq:zJ}.
 Note that we still need to find the quantity $J$ in \eqref{eq:zJ}. To do so, we observe that, by Lemma \ref{lem:representation_sol}, $J(x,t) = \exp(-\int_0^t  v_x ( \Phi(x, s), s) d s)$ is the  weak  solution of the following  auxiliary initial value problem:
\[ J_t + (J v)_x = 0 \; \te{ on } [0,1]^2,\quad \te{with }  J (x, 0 ) \equiv 1.\]
Hence, by \eqref{eq:zJ}, we have an analytic formulation for $z( \Phi ( x, t) , t)$:
\begin{equation}\label{eq:zJ2}
z( \Phi ( x, t) , t) 
=  \left( f_1 ( \Phi(x, 1) ) - f_0 (x) \right) \, \frac{J(x,t)}{\int_{0}^{1} J(x,s)  \dd s}.
\end{equation}
We also have an analytic formulation for $f( \Phi ( x, t) , t)$ given by~\eqref{eq:f-formula},
\begin{align}
f(\Phi(x,t),t) 
 &= \left(1-\eta(x,t) \right) \,  f_0( x) +  \eta(x,t) \, f_1 ( \Phi ( x, 1)), \quad \te{where } \eta(x,t) =  \frac{\int_0^t J(x,s) \dd s}{\int_0^1 J(x,s) \dd s}. \label{eq:f-formula2}
\end{align}
\smallskip

As a final step, through a change of coordinate, we obtain functions $f$ and $z$ at the $(x,t)$ coordinate (which are used for $v$), rather than the flow map coordinate $\left(\Phi(x,t), t\right)$.
At the discrete level, this will amount to an interpolation step between the points where $f$ and $z$ are computed along the flow, and the desired grid points.

Combining all steps above, we have defined a continuous operator 
$$
\mathcal{G}_1: v\mapsto \left(z,f \right),
$$
which solves~\eqref{eq:sub1} and~\eqref{eq:sub2}. Note that we have $(f,v,z)\in \mathcal{A}$ automatically if $(f,z) = \mathcal{G}_1(v)$. If we denote the $(f,v,z)$ from the previous step by $(f^{\text{old}},v^{\text{old}}, z^{\text{old}} )$, where $(f^{\text{old}},v^{\text{old}}, z^{\text{old}}) \in \mathcal{A}$, we then have
\[
A_{\kappa, \lambda, \veps}(f,v^{\text{old}},z) \leq A_{\kappa, \lambda, \veps}(f^{\text{old}},v^{\text{old}},z^{\text{old}}).
\]
The steps above yield functions $(f,v^{\text{old}},z)$ satisfying~\eqref{eq:v-BCs}, \eqref{eq:elz}, \eqref{eq:elf}, and~\eqref{eq:f-ICs} for the given $v^{\text{old}}$, but they do not necessarily satisfy~\eqref{eq:elv} in the Euler--Lagrange equations. Otherwise, we have found a set of solutions satisfying the first-order optimality conditions of~\eqref{eq:obj}. 

\subsubsection{From $f$ to $(v,z)$}\label{subsec:f2vz}
Given $f$ from the previous step, we consider the second sub-problem of finding the pair $(v^\te{new},z^\te{new})$ that minimizes the action \eqref{eq:obj} under the constraint that $(f,v,z) \in \mathcal A$. This 
again is a quadratic optimization problem under linear constraint: 
\begin{equation}\label{eq:sub4}
\min_{v,z} \frac12 \int_0^1 \int_0^1   \kappa v^2 + \lambda v_x^2  + \varepsilon v_{xx}^2 + z^2   \, \dd x \dd t, \quad \text{s.t.~}  (f,v,z) \in \mathcal{A}.
\end{equation}

Using the Euler--Lagrange equation \eqref{eq:elv} and the constraint $z = f_t + vf_x$ yields the following fourth order boundary value problem for $v$:
\begin{align}\label{eq:sub3}
\begin{split}
\varepsilon v_{xxxx} -\lambda v_{xx} + (\kappa  + |f_x|^2 ) v  & = -f_t \,f_x  \quad \te{on } (0,1)^2, \\
 v & = 0  \quad \te{ on } \{0,1\} \times (0,1), \\
 v_{xx} & = 0 \quad \te{ on }  \{0,1\} \times (0,1).
\end{split}
\end{align}
Given $f$ and the solution $v$, we also obtain $z$. Let us denote the solution operator by  $\mathcal{G}_2: f \mapsto \left(v^\te{new},z^\te{new} \right)$.

We note that 
\begin{equation}\label{eq:local_search}
A_{\kappa, \lambda, \veps}(f,v^{\text{new}},z^{\text{new}}) \leq A_{\kappa, \lambda, \veps}(f,v^\te{old},z).
\end{equation}
\medskip

Therefore, combining $\mathcal{G}_2$ with the previous step in Section~\ref{subsec:v2fz}, we have the following inequalities with respect to the action functional:
\[
A_{\kappa, \lambda, \veps}(f^{\text{new}},v^{\text{new}},z^{\text{new}}) \underbrace{\leq}_{(\text{$v^{\text{new}},z^{\text{new}}) = \mathcal{G}_2(f)$},\,f^{\text{new}} = f} A_{\kappa, \lambda, \veps}(f,v^{\text{old}},z)\underbrace{\leq }_{\text{$(f,z) = \mathcal{G}_1(v^{\text{old}})$}} A_{\kappa, \lambda, \veps}(f^{\text{old}},v^{\text{old}},z^{\text{old}}),
\]
where  $(f^{\text{new}},v^{\text{new}},z^{\text{new}}) , (f,v^{\text{old}},z), (f^{\text{old}},v^{\text{old}},z^{\text{old}}) \in \mathcal{A}$, all satisfying the constraints. If we define a new operator $\mathcal{G}$ by composing $\mathcal{G}_2 $ with $\mathcal{G}_1 $, i.e.,
\begin{equation}
    \mathcal{G} = \mathcal{G}_2 \circ \mathcal{G}_1: \mathcal{A} \mapsto \mathcal{A},
\end{equation}
it gives an update formula after which the action functional decays:
\begin{equation}\label{eq:energy_decay} 
A_{\kappa, \lambda, \veps}\left(    \mathcal{G}  (f,v,z) \right) \leq A_{\kappa, \lambda, \veps}(f,v,z).
\end{equation}
We can then repetitively applying $\mathcal{G}$ until finding a set of solution $ (f^*,v^*,z^*)$ where $  \mathcal{G}  (f^*,v^*,z^*) = (f^*,v^*,z^*)$. That is, $(f^*,v^*,z^*)$ is a fixed point of $\mathcal{G}$, while~\eqref{eq:energy_decay} indicates the contractivity of the fixed-point operator with respect to the action functional. It is easy to verify that $ (f^*,v^*,z^*) $ also solves the Euler--Lagrange equations~\eqref{eq:elv}-\eqref{eq:f-ICs}.

\begin{algorithm}
\caption{An iterative scheme for minimizing~\eqref{eq:obj}.\label{alg:picard-0}}
\begin{algorithmic}[1]
\State Given an initial guess $(f^{(0)},v^{(0)}, z^{(0)}) \in \mathcal{A}$, maximum number of iterations $N$, tolerance $\delta >0$.
\For{$n= 1$ to $N$} 
    \State Compute $(\tilde f, \tilde z ) = \mathcal{G}_1(v^{(n)} )$ with $\mathcal{G}_1$ described in Section \ref{subsec:v2fz} and set $f^{(n+1)} = \tilde f $.
    \State  Set $(v^{(n+1)}, z^{(n+1)} )= \mathcal{G}_2(f^{(n+1)})$ with $\mathcal{G}_2$ described in Section \ref{subsec:f2vz}.
    \If { $|A_{\kappa, \lambda, \veps}(f^{(n+1)},v^{(n+1)},z^{(n+1)} ) - A_{\kappa, \lambda, \veps}(f^{(n)},v^{(n)},z^{(n)} ) | < \delta $} 
    \State Return $\left( f^{(n+1)}, z^{(n+1)}  , v^{(n+1)} \right)$ and the minimum action value; \textbf{Break}. \EndIf
\EndFor
\end{algorithmic}
\end{algorithm}

\begin{algorithm}
\caption{An iterative scheme for minimizing~\eqref{eq:obj} with damping.\label{alg:picard}}
\begin{algorithmic}[1]
\State Given an initial guess $(f^{(0)},v^{(0)}, z^{(0)}) \in \mathcal{A}$, maximum number of iterations $N$, tolerance $\delta >0$.
\For{$n= 1$ to $N$} 
    \State Compute $(\tilde f, \tilde z ) = \mathcal{G}_1(v^{(n)} )$ and set $f^{(n+1)} = \alpha_1 \tilde f + (1-\alpha_1) f^{(n)}$, with $\alpha_1$ given by Algorithm~\ref{alg:btls}\\
    \hspace*{14pt} and $\mathcal{G}_1$ described in Section \ref{subsec:v2fz}.
    \State  Compute $(\widehat v, \widehat z ) = \mathcal{G}_2(f^{(n+1)} )$ and  set $(v^{(n+1)}, z^{(n+1)}  )= \alpha_2 (\widehat v, \widehat z ) + (1-\alpha_2) (v^{(n)}, z^{(n+\frac{1}{2})})$, with \\ 
     \hspace*{14pt} $\alpha_2$ given by Algorithm~\ref{alg:btls}
     and $\mathcal{G}_2$ described in Section \ref{subsec:f2vz}.
    \If { $A_{\kappa, \lambda, \veps}\left(f^{(n+1)},v^{(n+1)},z^{(n+1)} \right) > A_{\kappa, \lambda, \veps}\left(f^{(n)},v^{(n)},z^{(n)} \right) - \delta $} 
    \State Return $\left( f^{(n+1)}, z^{(n+1)}  , v^{(n+1)} \right)$ and the minimum action value; \textbf{Break}. \EndIf
\EndFor
\end{algorithmic}
\end{algorithm}

\begin{algorithm}
\caption{Back-tracking line search for the damping parameter.\label{alg:btls}}
\begin{algorithmic}[1]
\State Given the old iterate $(\bar f,\bar v,\bar z)$, the proposed new iterate $(\tilde f,\tilde v,\tilde z)$, the objective function $A_{\kappa, \lambda, \veps}$, and the maximum number of search steps $\widetilde N$. Set $\alpha = 1$, $\bar{A} = A_{\kappa, \lambda, \veps}(\bar f,\bar v,\bar z)$ and $\text{FLAG} = 0$.
\For{$i = 1$ to $\widetilde N$} 
\If{$A_{\kappa, \lambda, \veps}(f_\alpha, v_\alpha, z_\alpha) < \bar{A}$ where $(f_\alpha, v_\alpha, z_\alpha) = (1-\alpha)(\bar f,\bar v,\bar z) + \alpha(\tilde f,\tilde v,\tilde z)$}
\State {Return $\alpha$ and set $\text{FLAG} = 1$; \textbf{Break}.}
\EndIf
\State $\alpha \leftarrow \alpha/2$.
\EndFor
\If{$\text{FLAG} = 0$} \State Return $\alpha =0$. \EndIf 
\end{algorithmic}
\end{algorithm}

\begin{algorithm}
\caption{Searching for a path minimizing action \eqref{eq:obj} via different initializations. \label{alg:init}}
\begin{algorithmic}[1]
\State Given the maximum number of iterations $N \in \mathbb{N}^+$, $\delta >0$, $k_{max} \in \mathbb{N}$.
\For{$k = 0$ to $k_{max}$} 
    \State Use the prominence-based matching initialization described in Section~\ref{subsec:init} with parameter $k$ and  \\
    \hspace*{14pt} obtain $(f^{(0)},v^{(0)}, z^{(0)})$.
    \State Run Algorithm~\ref{alg:picard} with $N$ iterations and tolerance $\delta$. Obtain the minimum action value $\mathcal{J}(k)$.
\EndFor
\State Find $k^* = \text{argmin}\,\mathcal{J}(k)$.
\State Return the optimal path $(f,v,z)$ and its action value $\mathcal{J}(k^*)$ for initialization with parameter $k^*$.
\end{algorithmic}
\end{algorithm}

\subsection{The discrete scheme} \label{sec:num-discrete}
Next, we use a simple first-order numerical scheme to solve the two sub-problems discussed above. If the signals $f_0$ and $f_1$ are known to be smooth, higher-order discretization schemes would be recommended for better efficiency and accuracy. If the signals are discontinuous,  first-order schemes, on the other hand, are known to mitigate the Gibbs phenomenon~\cite{lax2006gibbs} that higher-order methods may suffer. 


To evaluate $(f,z) = \mathcal{G}_1(v)$,  we compute $z ( \Phi ( x, t), t )$ and $f ( \Phi ( x, t), t )$  using first-order numerical integration based on~\eqref{eq:zJ2} and~\eqref{eq:f-formula2}. We obtain $z(x,t)$ and $f(x,t)$ from $z ( \Phi ( x, t), t )$ and $f ( \Phi ( x, t), t )$ through first-order numerical interpolation in the $x$ variable alone. This step can be implemented in a  parallel fashion.

To evaluate $(v,z) = \mathcal{G}_2(f)$,  we have a fourth-order PDE for $v(\cdot, t)$  for every fixed $t$, with $v=0$ and $v_{xx} =0$ (if $\varepsilon \neq 0$) on the boundaries $\{0,1\}$. 

Consider a uniform mesh over the spatial domain $[0,1]$ and the time domain $[0,1]$. The  spatial spacing $\Delta x = 1/{N_x}$ and the time-domain spacing $\Delta t = 1/{N_t}$. 
Let ${\bf{v}}_j \in \mathbb{R}^{N_x+1}$ be a vector approximating $ [ v(0, j \Delta t),\ldots, v(i\Delta x, j \Delta t),\ldots, v(1, j \Delta t)]^\top$. The PDE~\eqref{eq:sub3} then becomes the following linear system under the finite-difference discretization,
\[
\left( A + \text{diag}\left( {\bf w}_j \odot {\bf w}_j \right) \right) \,{\bf{v}}_j = - \boldsymbol{\tau}_j \odot {\bf w}_j,
\]
where  $\odot$ denotes the Hadamard product, $\text{diag}({\bf x})$ denotes a diagonal matrix with elements of vector ${\bf x}$ being its diagonal entries,  $A\in \mathbb{R}^{(N_x+1) \times (N_x+1)}$ is given by
\begingroup\makeatletter\def\f@size{8}\check@mathfonts
\[
\begin{bmatrix}
1 &   &   &  &  &  &  \\ \\
-\frac{4\varepsilon}{\Delta x^4} - \frac{\lambda}{\Delta x^2}   & \kappa + \frac{2\lambda}{\Delta x^2} + \frac{5\varepsilon}{\Delta x^4} & -\frac{4\varepsilon}{\Delta x^4}- \frac{\lambda}{\Delta x^2} & \frac{\varepsilon}{\Delta x^4} & & & \\ \\
\frac{\varepsilon}{\Delta x^4} & -\frac{4\varepsilon}{\Delta x^4} - \frac{\lambda}{\Delta x^2}   & \kappa + \frac{2\lambda}{\Delta x^2} + \frac{6\varepsilon}{\Delta x^4} & -\frac{4\varepsilon}{\Delta x^4}- \frac{\lambda}{\Delta x^2} & \frac{\varepsilon}{\Delta x^4} &  &  \\
 &   &   &   &   &   &  \\
 &    \ddots &  \ddots &  \ddots &  \ddots &  \ddots  & \\
 &  &   \frac{\varepsilon}{\Delta x^4} & -\frac{4\varepsilon}{\Delta x^4}- \frac{\lambda}{\Delta x^2}   & \kappa +  \frac{2\lambda}{\Delta x^2} +\frac{6\varepsilon}{\Delta x^4}  & -\frac{4\varepsilon}{\Delta x^4} - \frac{\lambda}{\Delta x^2}  & \frac{\varepsilon}{\Delta x^4}  \\ \\
&   & &   \frac{\varepsilon}{\Delta x^4} & -\frac{4\varepsilon}{\Delta x^4}- \frac{\lambda}{\Delta x^2}   & \kappa + \frac{2\lambda}{\Delta x^2} + \frac{5\varepsilon}{\Delta x^4}  & -\frac{4\varepsilon}{\Delta x^4} - \frac{\lambda}{\Delta x^2}  \\ \\
 &  &   &  &  &  & 1 \\
\end{bmatrix},
\]
\endgroup
and ${\bf{w}}_j, {\boldsymbol{\tau}}_j  \in \mathbb{R}^{N_x+1}$, approximating  $f_x(\cdot, t_j)$ and $f_t(\cdot, t_j)$, respectively, are given by
\begin{align*}
{\bf{w}}_j &=  \frac{1}{\Delta x}\begin{bmatrix} 0 &    \left( f(x_3,t_j) - f(x_2,t_j) \right) &  \ldots & \left( f(x_{N_x+1},t_j) - f(x_{N_x},t_j)\right) & 0 \end{bmatrix}^\top,\\
{\boldsymbol{\tau}}_j &= \frac{1}{\Delta t}\begin{bmatrix} 0 &    \left( f(x_2,t_{j+1}) - f(x_{2},t_j) \right) &  \ldots & \left( f(x_{N_x},t_{j+1}) - f(x_{N_x},t_{j})\right) &  0 \end{bmatrix}^\top,
\end{align*}
with $f(x_i,t_j) = f(i \Delta x, j \Delta t)$. We remark that the first and last elements of the right-hand side are set to be zero while the first and last rows of the left-hand side are also modified. These two linear equations are to enforce $v=0$ on the boundary; see~\eqref{eq:sub3}. The fact that $v_{xx} = 0$ on the boundary implies that the ghost points $v(-\Delta x, \cdot) = - v(\Delta x, \cdot)$, and $v(1+\Delta x, \cdot) = - v(1-\Delta x, \cdot)$, which is used in the second and the $N_x$-th rows of $A$. The remaining $N_x-3$ equations in the linear system are to enforce the linear PDE~\eqref{eq:sub3}. Note that with a given $f$, we can solve for $\{{\bf v}_j\}$ in parallel for all time $\{t_j\}$, or construct a large sparse linear system with respect to ${\bf v} = [{\bf v}_1^\top \ldots {\bf v}_{N_t+1}^\top]^\top$ in one single sparse linear solve. When $f_0, f_1$ are smooth, it is preferable to use the central difference method to obtain ${\boldsymbol{\tau}}_j$. 

Now we have two steps: $\mathcal{G}_1: v \mapsto (f,z)$ through first-order numerical integration and interpolation, 
and $\mathcal{G}_2: f\mapsto (v,z)$ through a linear PDE solver. On the continuous level, we have a monotonic energy decay based on~\eqref{eq:energy_decay}. Algorithm~\ref{alg:picard-0} combines these two steps to iteratively find approximate local minimizers of \eqref{eq:obj}.

While this algorithm works well for many signals, on the discrete level, due to numerical errors from interpolation, integration, and the PDE solver, \eqref{eq:energy_decay} may not hold for every iteration of the fixed-point update. To ensure that the \emph{discretized} variational problem still has a monotonically decaying action functional, we can introduce damping parameter $\alpha_1$ and $\alpha_2$, similar to the step size in a gradient descent algorithm; see Algorithm~\ref{alg:picard} for details. We remark that most of the time, $\alpha_1= \alpha_2 = 1$ is sufficient to ensure decay of the action value, which is the case on the continuous level. Only when it is close to the target solution, and the numerical errors dominate, one may observe that the action value no longer monotonically decreases when iterating between the two sub-problems. There damping helps obtain more accurate minima. We use back-tracking line search with a shrinking factor $1/2$ to find a proper damping coefficient; see details in Algorithm~\ref{alg:btls}.

\subsection{Finding a good initialization}\label{subsec:init}
We remark that the fixed points of $\mathcal{G}$ solve the Euler--Lagrange equations~\eqref{eq:elv}-\eqref{eq:f-ICs}, but they may not minimize~\eqref{eq:obj} over~\eqref{eq:adm} due to the possible local minima. In other words, the fixed-point operator $\mathcal{G}$ has multiple fixed points. One can start with different initial guesses and investigate the convergence behavior while choosing the solution with the smallest objective function.

We propose two different types of initial guesses.
\begin{enumerate}
    \item \textbf{Zero-velocity initialization.} We set $v^{(0)} (x,t) \equiv 0$, and compute  $(f^{(0)},z^{(0)}) = \mathcal{G}_2(v^{(0)})$. Note that in this case, we have $f^{(0)}(x,t) = (1-t) f_0(x) + t f_1(x)$ and $z^{(0)}(x,t) =  f_1(x) - f_0(x)$.
    \item \textbf{Prominence-matching initialization.} 
    We expect that if there exists a path with action that is substantially smaller than linear interpolation, that path will match nearby large peaks. We observe that matching the tallest peaks is not stable under perturbations as there may be spurious nearby peaks, for example, due to oscillations as in Figure \ref{fig:HFsignal165-Yunan}. For this reason, we use a more stable notion of how large the peaks are. Namely, we use the notion of \emph{prominence} coming from topography, which describes how large the peaks are compared to their surroundings. The prominence of a signal is defined as the least drop in height necessary to get to another local maximum with a higher value.
    Consider a positive integer $k$. For the given $f_0$ and $f_1$, we each select $k$ local maxima with the largest $k$ prominence.   The location of the local maxima are denoted by $\{x_i\}$ and $\{y_i\}$, $1\leq i \leq k$, respectively. We then construct a map $T(x)$ such that $T(x_i) = y_i$ for each $i$, and $T(0) = 0$, $T(1) = 1$. We use linear interpolation to define its function value for $x \in (0,1)\setminus \{x_i\}$. We then set the initial velocity to be $v^{(0)}(x,t) = T(x)-x$,  which is constant in time, and use $(f^{(0)},z^{(0)}) = \mathcal{G}_1(v^{(0)})$ as the initial guess for $f$ and $z$. Also, if the minima in $f_0$ and $f_1$ are more significant than the maxima, one can also initialize by matching the prominence of $-f_0$ and $-f_1$.
\end{enumerate}
The zero-velocity initialization is equivalent to a degenerate case of the prominence-matching initialization where $k=0$. In this scenario, we have $T(x) = x$ through linear interpolation between $T(0) = 0$ and $T(1) = 1$. Later, we will refer to the ``zero-velocity initialization''  as $k=0$.
We also comment that the initialization is different for various integer $k$, which may lead to different convergence behavior and local minima of the action functional~\eqref{eq:obj}. We suggest trying for a few $k$ values. While there are many variants on how one can incorporate the prominence-matching initialization into the optimization scheme, we outline an Algorithm~\ref{alg:init} as an example, which we use to produce the numerical results in Section~\ref{sec:experiments}. To be more efficient, one can use fewer iterations when searching for an initialization compared to running the entire Algorithm~\ref{alg:picard-0} or Algorithm~\ref{alg:picard} to find the minimizer of the optimization problem.



\begin{figure}
    \centering 
       \subfloat[Large bumps are much bigger than  small ones; Horizontal transport dominates.]{
           \includegraphics[width = 0.5\textwidth]{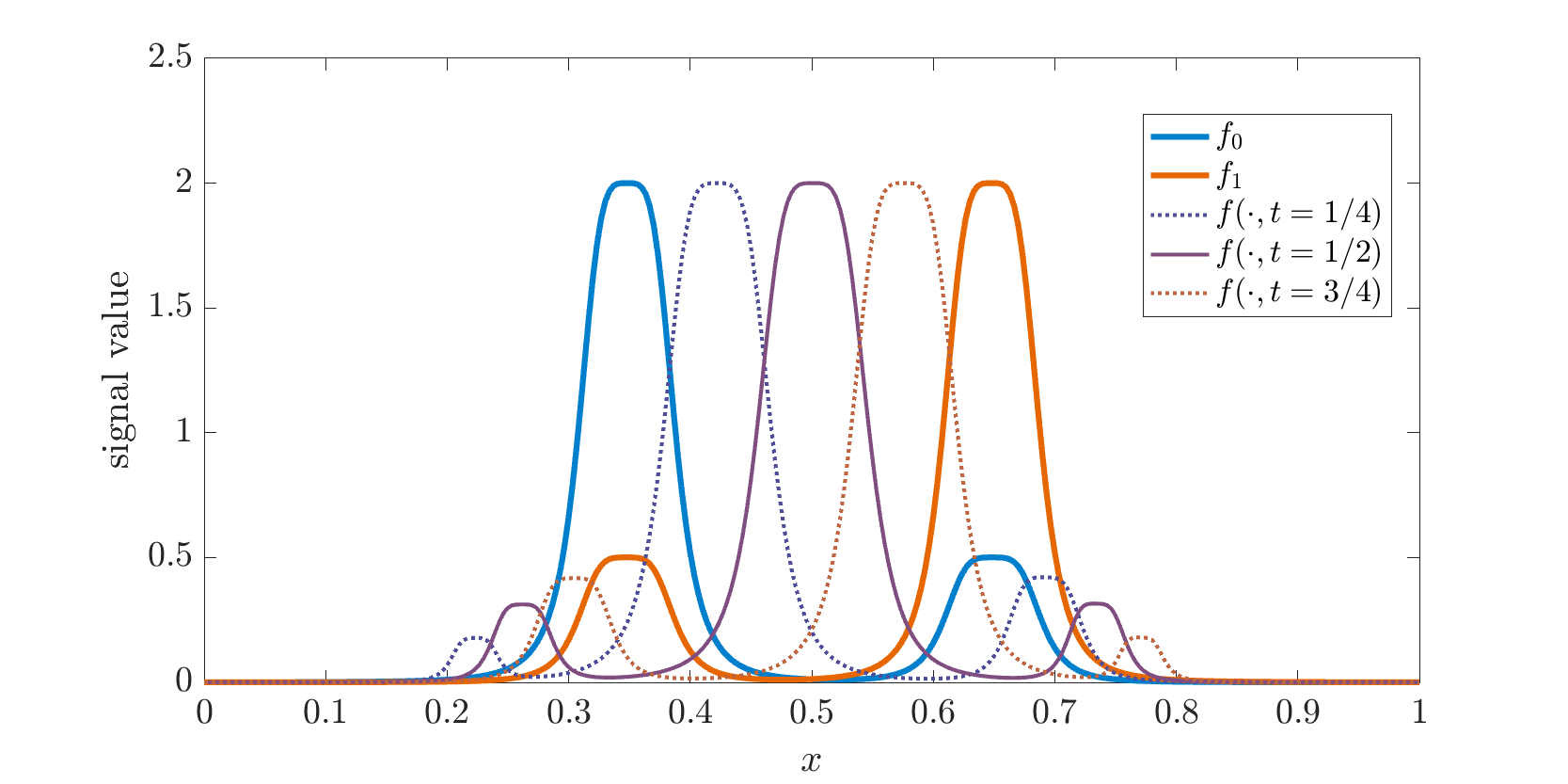}\label{fig:161} }
        \subfloat[Bumps are of comparable sizes; Vertical change dominates.]{
           \includegraphics[width = 0.5\textwidth]{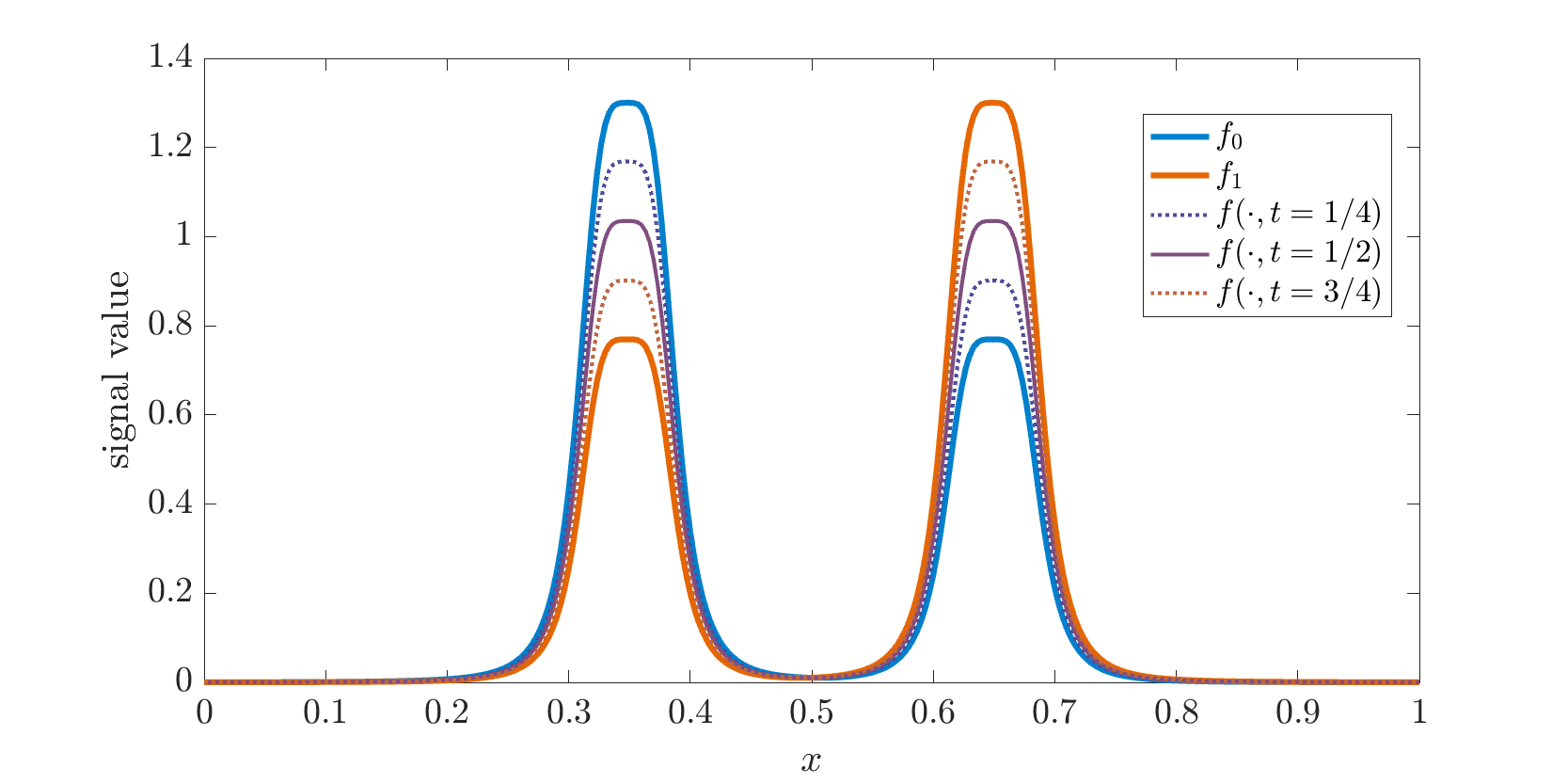}\label{fig:162}}
       \\
    \subfloat[For appropriate ratio of bump heights, both dominant transport mechanisms produce the same action.]{
           \includegraphics[width = 0.5\textwidth]{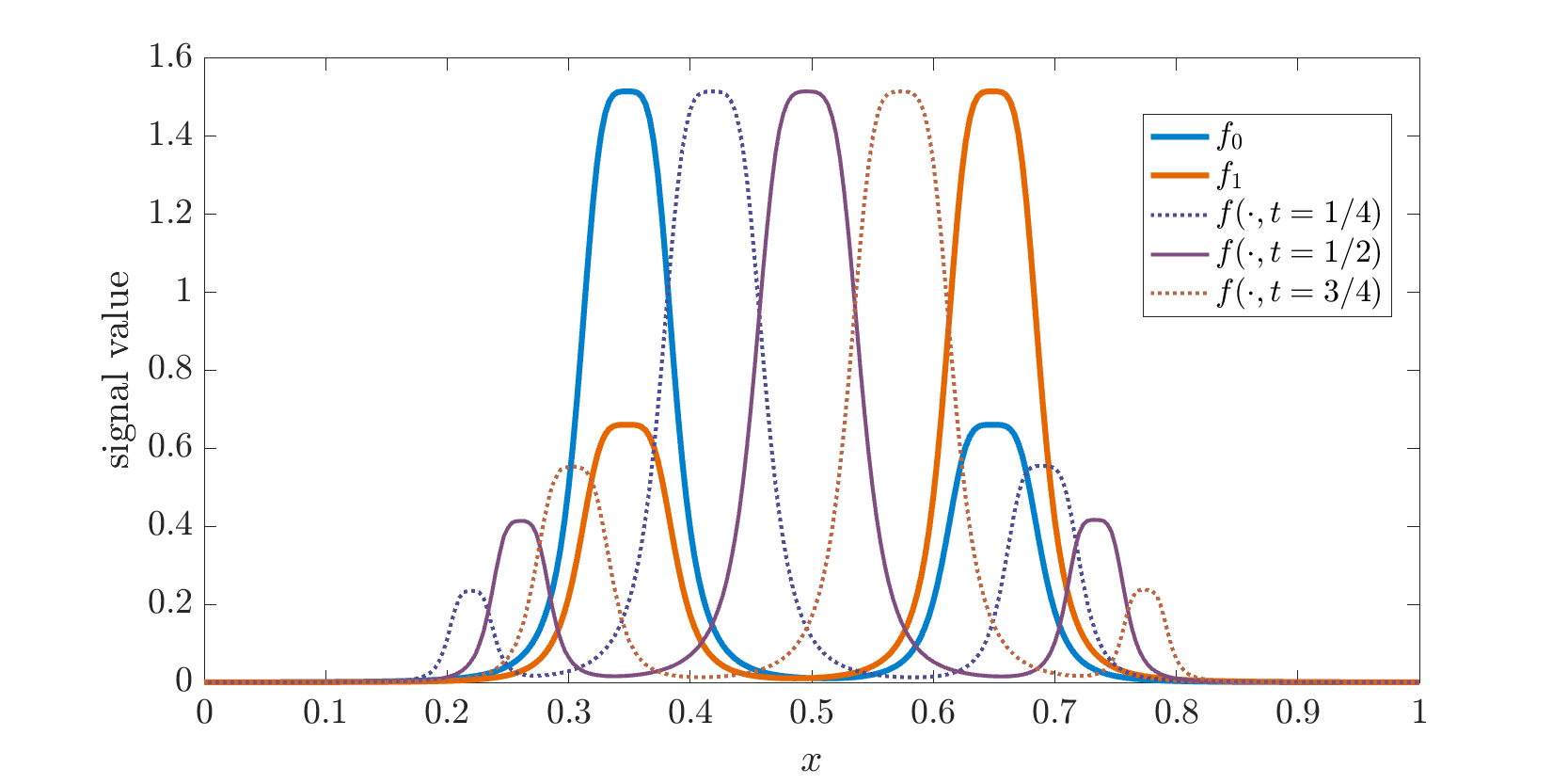}
           \includegraphics[width = 0.5\textwidth]{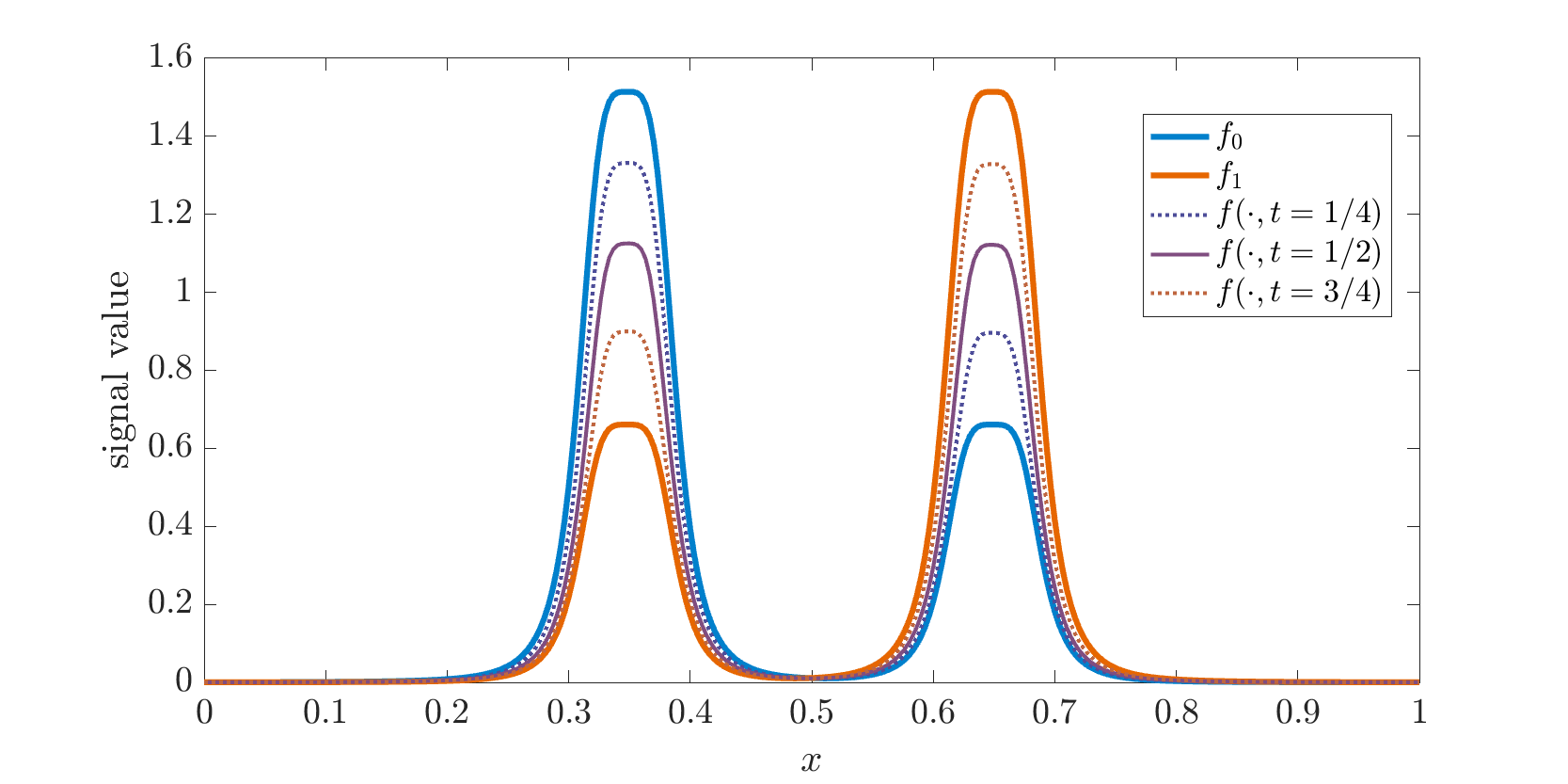}
       \label{fig:163&164}}
    \caption{Example of the nonuniqueness of the length-minimizing geodesic. We set $\kappa = 0.02$, $\lambda = 0.001$, and $\veps = 0.002$. We use $300$ spatial and $290$ time intervals.}
    \label{fig:twobumps-Yunan} 
\end{figure}

\begin{figure}
    \centering
    \subfloat[geodesics based on the HV geometry]
    {\includegraphics[width = 0.75\textwidth]{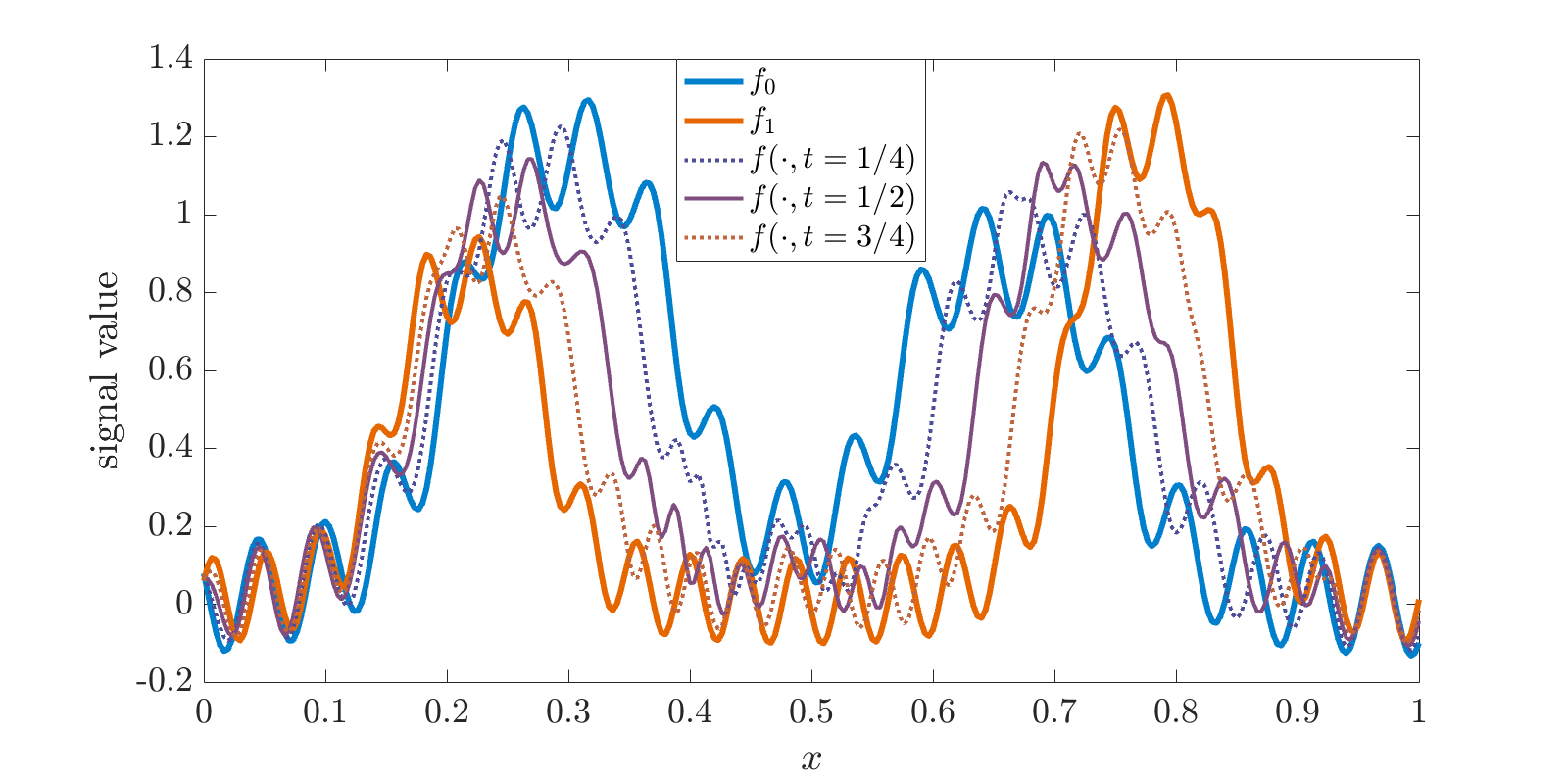}\label{fig:165-geo}}\\
    \subfloat[flow map for the initial velocity (red trajectories denote the initial matching of the most prominent two peaks between $f_0$ and $f_1$)]{\includegraphics[width = 0.75\textwidth]{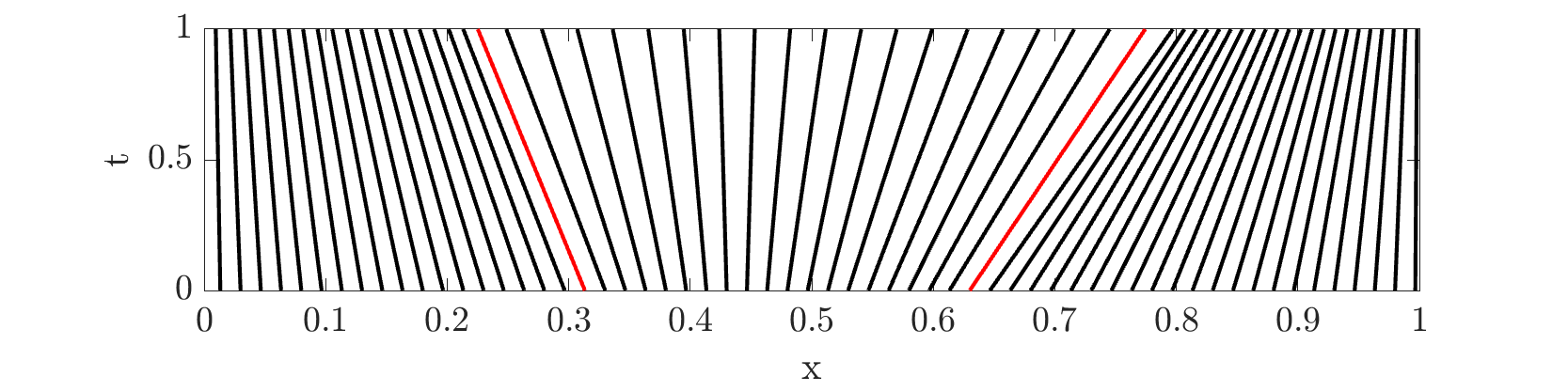}\label{fig:165-init}}\\
   \subfloat[flow map for the optimal velocity]{\includegraphics[width = 0.75\textwidth]{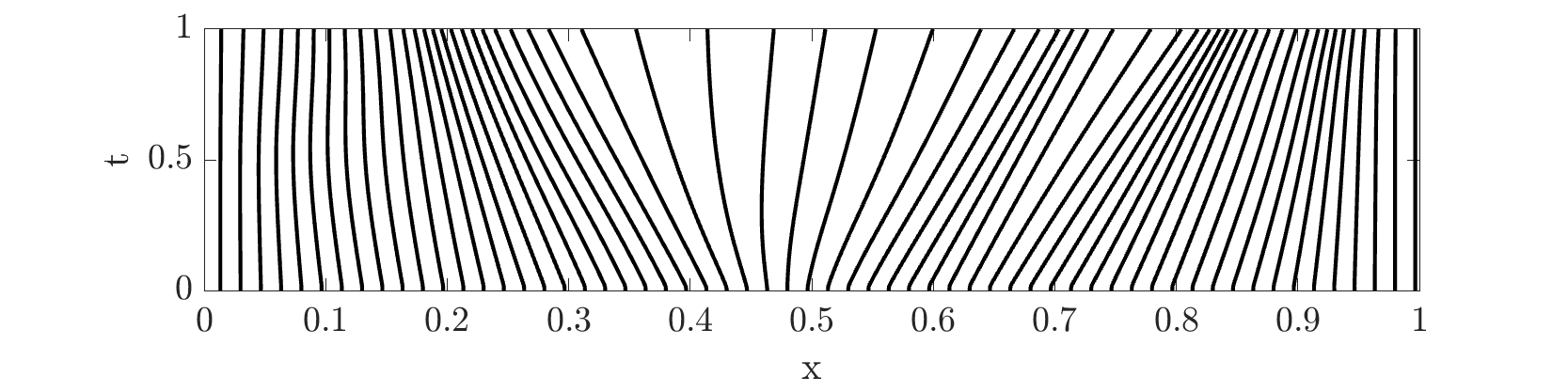}\label{fig:165-final}}
     \caption{We use prominence-matching initialization to find the action-minimizing path. We show in (A) the initial signal (blue), final signal (red), and signals at $t=0.25$, $t=0.5$, and $t=0.75$ along the computed geodesic for hyperparameters $\kappa = 10^{-3}$, $\lambda = 5\times 10^{-5}$, and $\veps = 2.5\times 10^{-5}$. We used $300$ space intervals and $290$ time intervals. The flow for the initial velocity is shown in (B), and the flow for the optimal velocity is shown in (C).}
    \label{fig:HFsignal165-Yunan}
\end{figure}

\section{Numerical Experiments}\label{sec:experiments}
In this section, we present a few  examples illustrating the geodesic using the HV geometry\footnote{The codes based on the numerical scheme described in Section~\ref{sec:scheme} that reproduce these examples can be found at \url{https://github.com/yunany/Compute-HV-distance-between-signals.git}.}. Throughout this section, we plot the source signal $f_0$ in blue and the target signal $f_1$ in red, while their barycenter under the HV geometry is shown using the color purple. We use dashed lines to indicate the signals at $t=\frac14$ and $t=\frac34$.

\begin{figure}
    \centering
    \includegraphics[width = 0.6\textwidth]{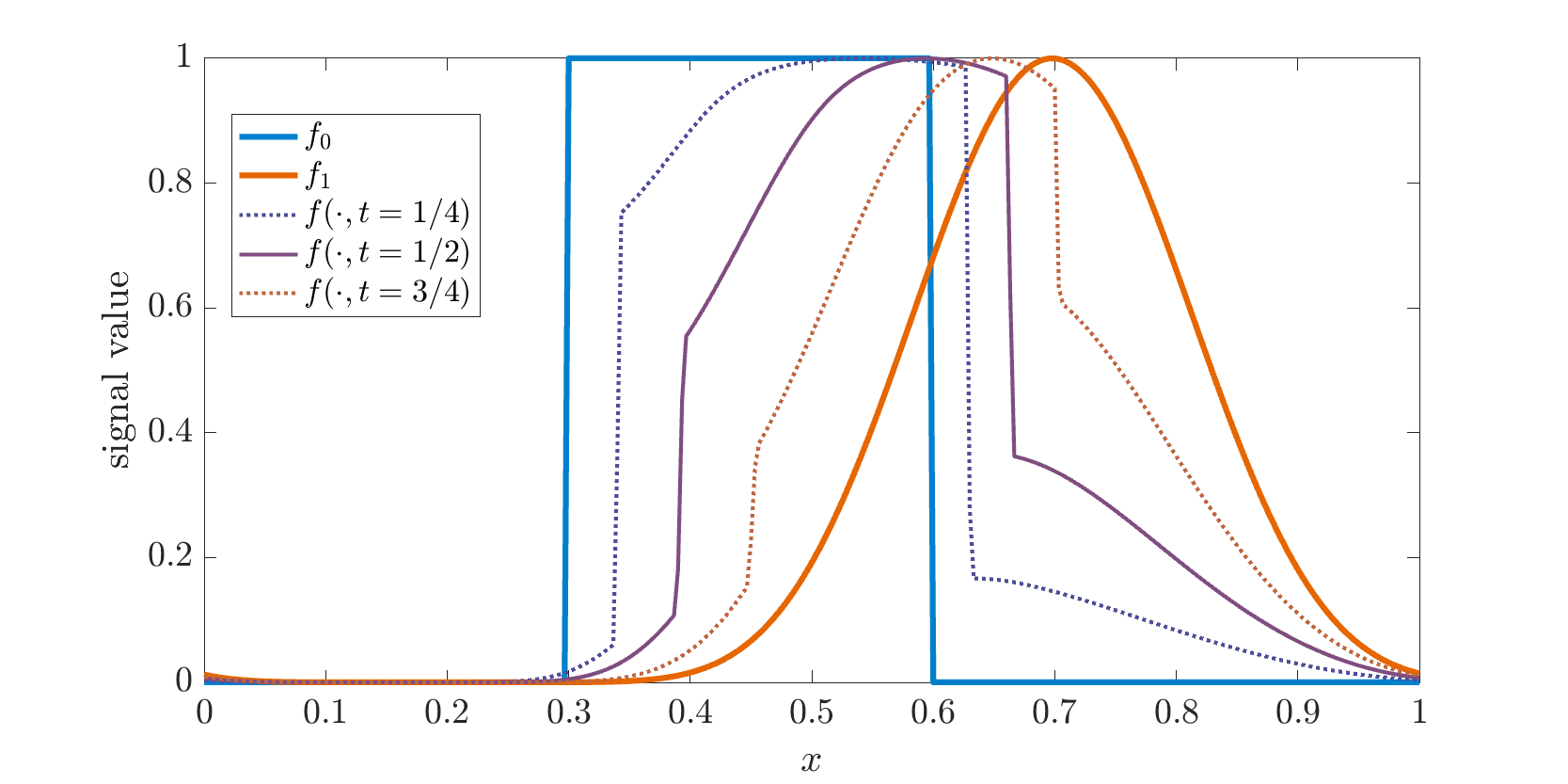}
    \caption{Algorithm allows for non-smooth data. Here, we use the hyperparameters $\kappa=0.02$, $\lambda=0.001$, $\veps=0.002$,  $300$ space intervals, and $290$ time intervals.}
    \label{fig:box_signal14}
\end{figure}

\begin{example} \emph{Nonuniqueness of minimizing geodesics.}
We consider an example where the source and target signals have two bumps. We use $300$ spatial and $290$ time intervals to discretize the space-time domain. For all plots in Figure~\ref{fig:twobumps-Yunan}, we set the hyperparameters $\kappa = 0.02$, $\lambda = 0.001$, and $\veps = 0.002$. In Figure~\ref{fig:161}, each signal has a large bump and a small bump. The large bumps are much bigger. The geodesic is dominated by horizontal transport. In Figure~\ref{fig:162}, the two bumps in each signal are of comparable size while their locations remain the same. The geodesics, in this case, is dominated by vertical changes instead of horizontal transport in the previous example. Finally, by adjusting the ratio of the bump heights while fixing their locations, we find a scenario where horizontal transport and vertical changes result in the same action value; see Figure~\ref{fig:163&164}. That is, we found local minimizers, which we believe to be global, where the paths have the same action. Thus, we believe the geodesics are not unique, indicating that the signal space is, at least partly, positively curved.
\end{example}


\begin{example}\emph{Bumps with high frequency perturbations.}  
Here we use the prominence-matching initialization introduced in the previous section to find a good starting point for our iterative scheme to converge to the global minimizer of the action function~\eqref{eq:obj}. In Figure~\ref{fig:165-geo},  we show the geodesics for hyperparameters $\kappa = 10^{-3}$, $\lambda = 5\times 10^{-5}$, and $\veps = 2.5\times 10^{-5}$. We used $300$ space intervals and $290$ time intervals. In Figure~\ref{fig:165-init}, we show the flow map based on the prominence-matching initial velocity, which maps the source signal's two most prominent peaks to the target signal's two most prominent peaks. The trajectories matching the peaks are indicated in red, while the remaining trajectories are obtained by linear interpolation.  
The flow map with respect to the final converged velocity is shown in Figure~\ref{fig:165-final}, which is relatively close to Figure~\ref{fig:165-init}. 
\end{example}

\begin{example} \emph{Signal with discontinuities.}
In Figure~\ref{fig:box_signal14}, we consider a  non-smooth source signal and a smooth target signal. Here, we use the hyperparameters $\kappa=0.02$, $\lambda=0.001$, and $\veps=0.002$. We used $300$ space intervals and $290$ time intervals.  The finite difference discretization on a static mesh implicitly regularizes the signal $f_0$. Since we use a first-order numerical scheme, the computed solutions do not suffer from the Gibbs phenomenon as reflected by the geodesic in Figure~\ref{fig:box_signal14}. Moreover, we also observe the gradual transition between the discontinuous feature and the discontinuous feature both in horizontal and vertical directions.
\end{example}


\begin{example} \emph{Growth and expansion.} 
In Figure~\ref{fig:c5-Yunan}, we consider a ``growth'' example where the target signal is much bigger in width and height than the source signal. Here, we use the hyperparameters $\kappa=0.2$, $\lambda=0.01$, and $\veps=0.02$. We used $300$ space intervals and $290$ time intervals. We comment that this example is somewhat sensitive to the choice of hyperparameters, which directly affects the location of the barycenter between $f_0$ and $f_1$ under the HV geometry.
\end{example}
\begin{figure}
    \centering
    \includegraphics[width = 0.6\textwidth]{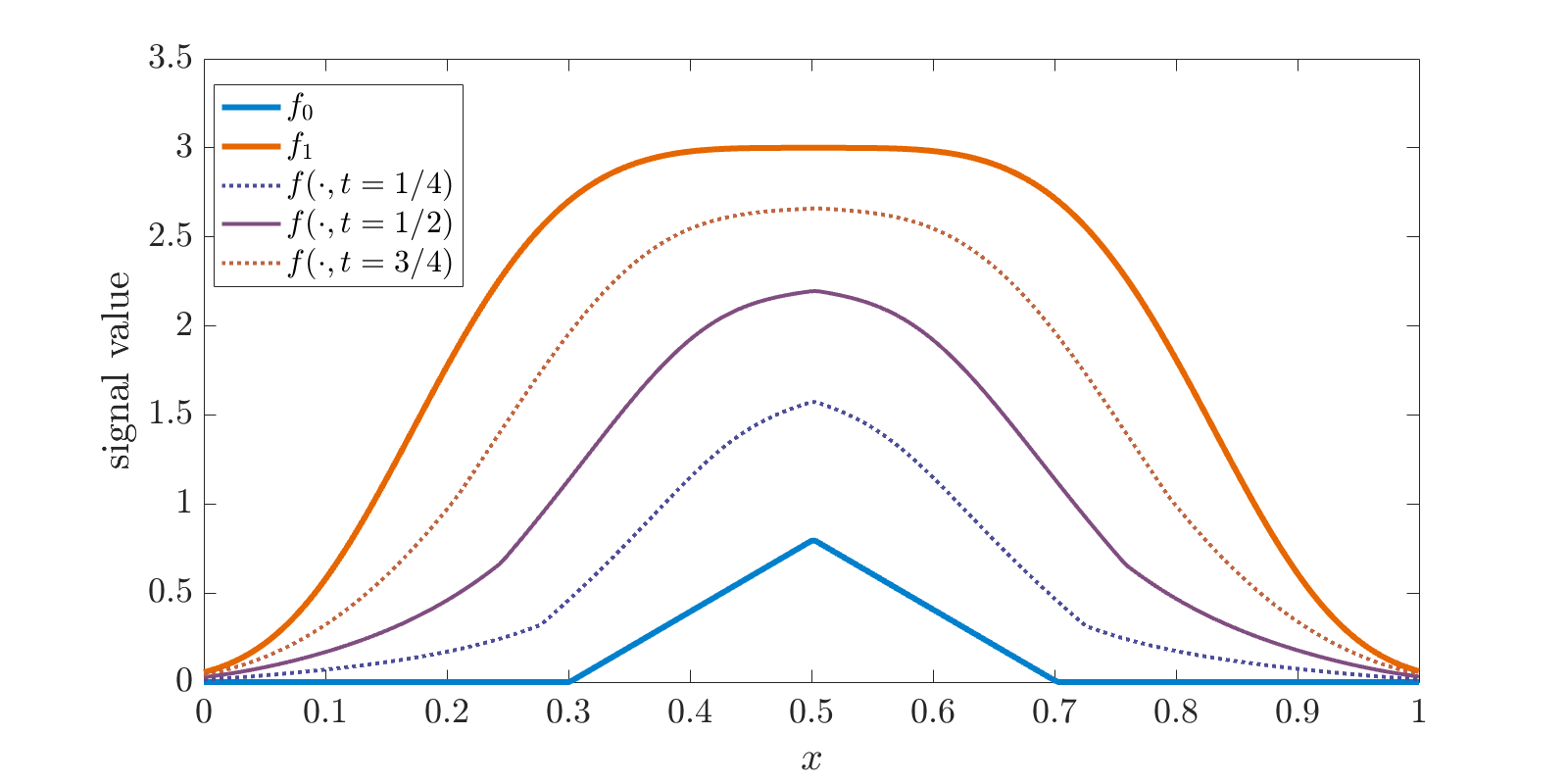}
    \caption{The $d_{HV}$ geodesic balances horizontal expansion and vertical growth.  We use $\kappa = 0.2$, $\lambda = 0.01$, $\veps = 0.02$, $300$ space intervals, and $290$ time intervals.}
    \label{fig:c5-Yunan}
\end{figure}

\begin{figure}
    \centering
    \includegraphics[width = 0.6\textwidth]{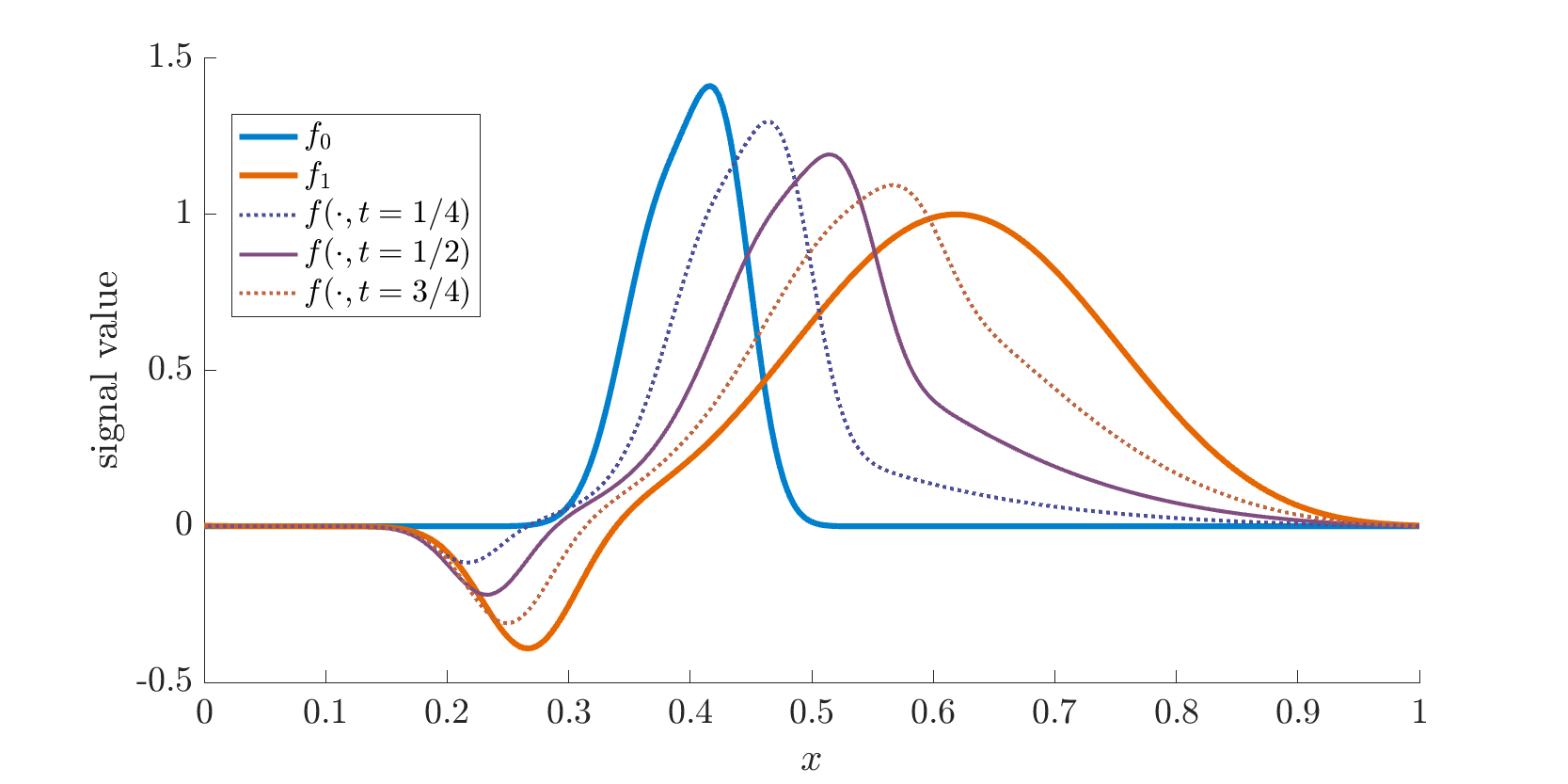} \\
    \includegraphics[width = 0.6\textwidth]{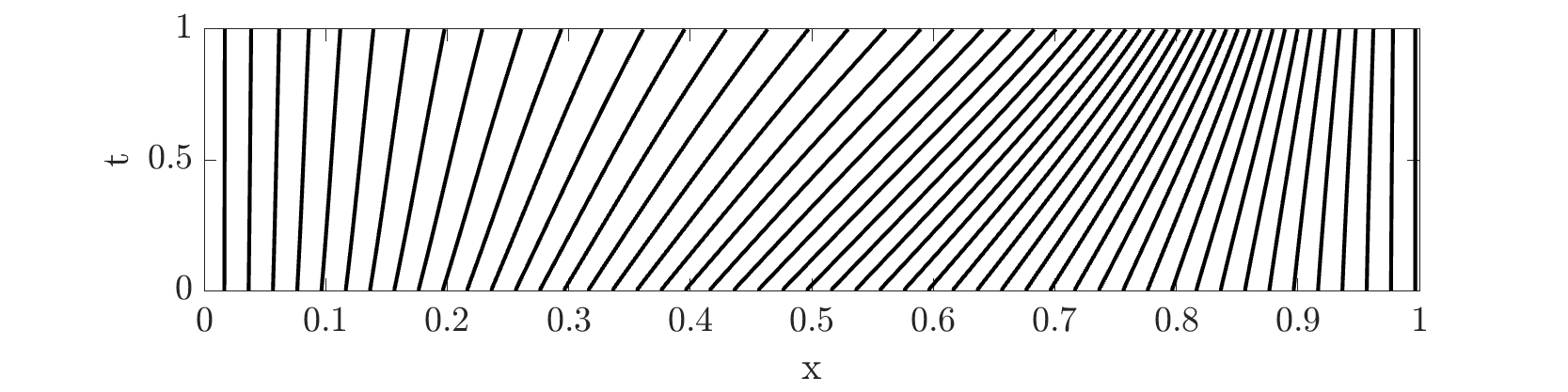}
    \caption{We show the initial signals along the computed geodesic for hyperparameters given by formula~\eqref{eq:paramaters} with $L=0.3$, $W = 0.4$ and $H=0.4$. The bottom image shows the flow of the optimal velocity $v$. 
    }
    \label{fig:signal8_Yunan}
\end{figure}

\begin{figure}
    \centering
    \subfloat[original signals]{\includegraphics[width = 0.5\textwidth]{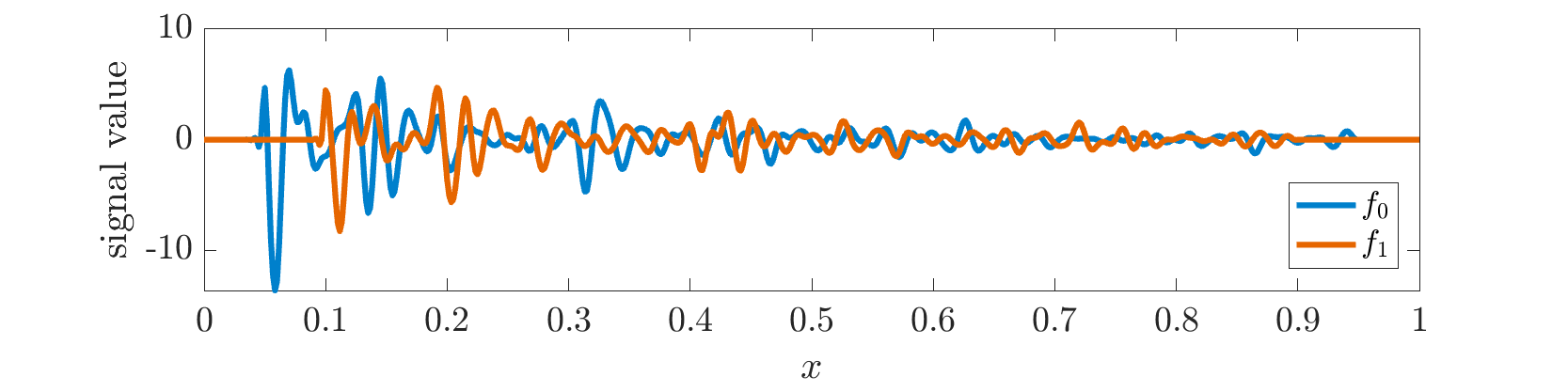}\label{fig:seismic1}}
    \subfloat[normalized signals $\hat{f}_0$ and $\hat{f}_1$ for OT]{\includegraphics[width = 0.5\textwidth]{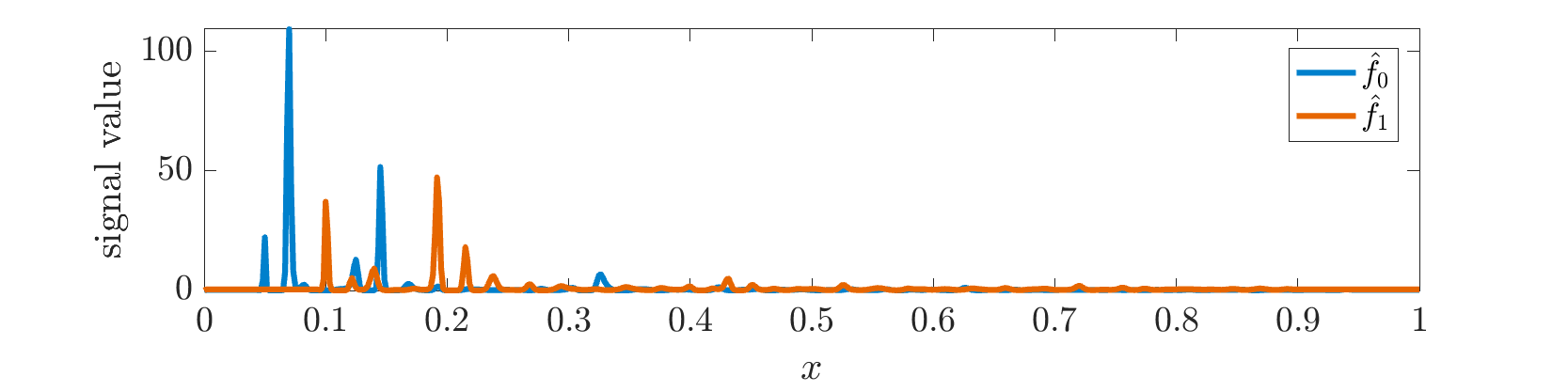}\label{fig:seismic2}}\\
    \subfloat[the optimizer $f(x,t)$ based on the HV geometry]{\includegraphics[width = 1\textwidth]{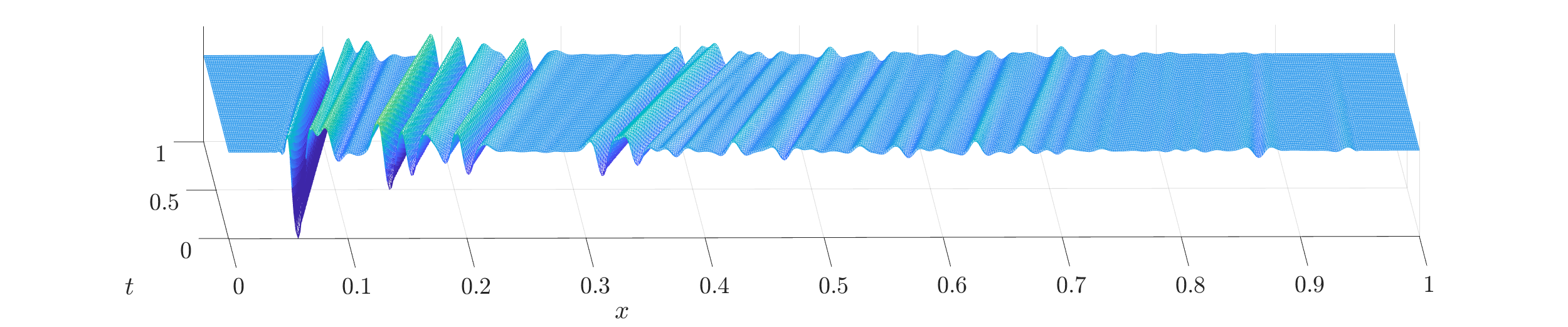}\label{fig:seismic3}}\\
    \subfloat[OT displacement interpolation between $\hat{f}_0$ and $\hat{f}_1$]{\includegraphics[width = 1\textwidth]{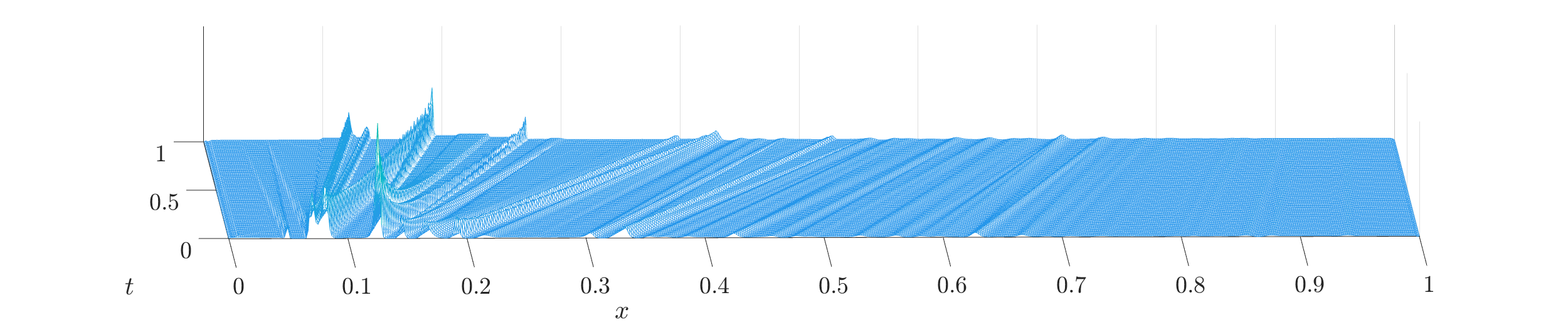}\label{fig:seismic4}}
    \caption{We compare the matching of two 1D seismic signals by the flow map induced by the HV geometry and optimal transportation with a quadratic cost (applied to normalized signals). (A): the original signal used in the HV geometry; (B): the normalized signal to satisfy the OT requirements; (C) optimizer $f(x,t)$ from the HV geometry; (D) the displacement interpolation between $\hat{f}_0$ and $\hat{f}_1$ based on the optimal transport map.}
    \label{fig:seismic}
\end{figure}

\begin{figure}
\centering
    \includegraphics[width = 0.80\textwidth]{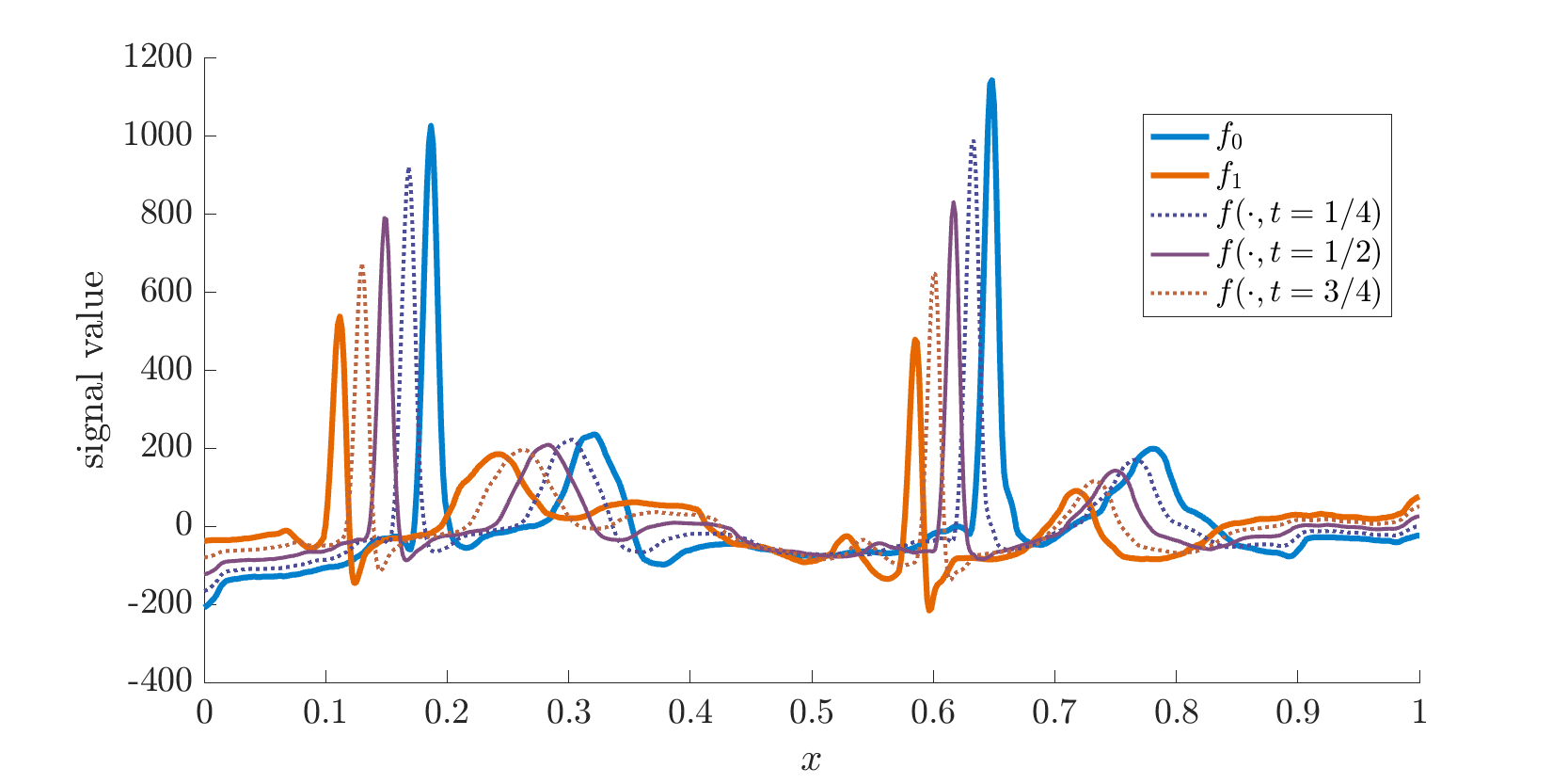}
    \caption{Shown is the geodesic between two ECG signals under the HV geometry. We note that the horizontal transform dominates for some features, while some parts are matched by vertically moving the graph. We set $\kappa$, $\lambda$, and $\veps$ based on~\eqref{eq:paramaters} with $L = W = 0.1$ and $H = 300$. We use $600$ space intervals and $150$ time intervals.}
    \label{fig:ECG}
\end{figure}

\begin{example} 
In Figure~\ref{fig:signal8_Yunan}, we compare two signed signals with a single bump where the bumps' widths and locations do not agree. We compute geodesic between them under the HV geometry for the chosen hyperparameters decided by the parameter estimate \eqref{eq:paramaters} discussed in Section~\ref{subsec:para}. 
Roughly estimate that $L=0.3$, $W = 0.4$ and $H=0.4$, which yields $\kappa=0.1$, $\lambda=0.01$ and $\veps=0.005$. We note that a wide set of parameters would have produced similar geodesics. 
The path is discretized using $300$ spatial and $290$ time intervals. The flow map corresponding to the optimal velocity $v$ is plotted at the bottom of Figure~\ref{fig:signal8_Yunan}, where one can observe the transport feature mapping the peak of the source signal (blue) to the peak of the target signal (red). 
\end{example}

\begin{example} \emph{(Seismic signals)}
Using optimal transportation (OT) for seismic applications has faced difficulties from the constraints that the signals should be nonnegative with equal total mass~\cite{engquist2019seismic,engquist2022optimal}. In this example, we test the proposed HV geometry for comparing synthetic seismic signals shown in Figure~\ref{fig:seismic1}. To compare them using OT, we may normalize the signals first so that they are nonnegative with equal total mass; see Figure~\ref{fig:seismic2} for the normalized signals $\hat{f}_0, \hat{f}_1$, squared and then scaled to integrate to one~\cite{engquist2019seismic}. Figures~\ref{fig:seismic3} and~\ref{fig:seismic4} show the HV and OT geometry velocity flow maps, respectively. We use a quadratic cost function for OT. For the HV geometry, we set $H = 2$, $L = 0.1$ and $W=0.02$ for parameters in~\eqref{eq:paramaters} presented in Section~\ref{subsec:para}. Note that in the classic OT, all mass has to be transported through the monotonic map $T$ such that $\int_0^x \hat{f}_0(y)\dd y = \int_0^{T(x)} \hat{f}_1(y) \dd y$. This may lead to mass being transported far away and unevenly, as illustrated in Figure~\ref{fig:seismic4}. The proposed HV geometry not only can handle signed signals naturally, avoiding the artifacts by preprocessing the signal but also enforces regularity to the velocity. 
\end{example}

\begin{example}
Finally, we consider a real-world example. We compare two heartbeats from the ECG database PhysioNet 2017 Challenge \cite{8331486,doi:10.1161/01.CIR.101.23.e215}; see Figure~\ref{fig:ECG}. The geodesic is computed in the space of signals according to the HV geometry for hyperparameters given by~\eqref{eq:paramaters} below in Section~\ref{subsec:para}, with $L=W =0.1$ and $H=300$, which one estimates from the given data.  We use $600$ space intervals and $150$ time intervals. We note that large features (R-peaks and T-waves) are matched in a desirable way via horizontal transport, while perturbations of small amplitude are matched via a vertical adjustment. 
This illustrates the benefits of the HV geometry.
\end{example}

\subsection{Parameter selection.}\label{subsec:para}
An important element in using HV geometry to analyze signals is how to select the parameters $\kappa, \lambda$, and $\veps$. This depends on the length scales present in the data. Here we give a simple rule for selecting the parameters based on the scaling properties of the distance; see Proposition \ref{prop:properties}. Let $H$ be the average vertical variation in the data, $W$ be the typical width of features in the data, and $L$ be the maximum horizontal distance between the features to be matched. Then we suggest using
\begin{equation} \label{eq:paramaters}
\kappa = 0.01 \frac{H^2}{L^2}, \quad \lambda = 0.02 H^2, \;\te{ and }\;\, \veps = 0.2 H^2 W^2.
\end{equation}
As we mentioned, the scaling of the parameters respects the invariances of the distance. The real number coefficients ($0.01$, $0.02$, $0.2$) are based on numerical experiments with different signal types. 

We note that given a data set $\mathcal F = \{f_1, \dots, f_n\}$ a good suggestion for $H$ would be the typical $L^2$ distance between the signals: 
\[ H^2 = \frac{1}{n^2} \sum_{i} \sum_{j} \| f_i  - f_j\|_{L^2}^2.  \]
We also note that $H^2$  is twice the variance and can thus  be computed as a sum over one index:
\[ H^2 = \frac{2}{n} \sum_i \| f_i \|_{L^2}^2  - 2 \,
\left\| \frac{1}{n} \sum_i f_i \right\|_{L^2}^2. \]
We remark that, for most signals, the outcome is not very sensitive to the parameters.


\section*{Acknowledgements}
We thank Jianming Wang for providing a key idea for the result in Section \ref{sec:deg}. 
We are grateful to Katy Craig for stimulating discussions. 
DS and RH are grateful to NSF for support via grant DMS 2206069. This work was done in part while DS and YY were visiting the Simons Institute for the Theory of Computing in Fall 2021. YY acknowledges support from Dr.~Max R\"ossler, the Walter Haefner Foundation and the ETH Z\"urich Foundation.

\bibliographystyle{siam}
\bibliography{HVbib}

\end{document}